\documentclass[onefignum,onetabnum]{siamart171218}
\usepackage{algorithm}
\usepackage{overpic}
\usepackage{algpseudocode}
\usepackage{tikz}
\usepackage{pgfplots}
\usepackage{appendix}
\pgfplotsset{compat=1.18}
\usepackage{booktabs}

\usepackage{lipsum}
\usepackage{amsfonts}
\usepackage{graphicx}
\usepackage{epstopdf}
\usepackage{amssymb}

\newtheorem{assump}{Assumption}
\newtheorem{theo}{Theorem}
\newsiamthm{lem}{Lemma}
\ifpdf
  \DeclareGraphicsExtensions{.eps,.pdf,.png,.jpg}
\else
  \DeclareGraphicsExtensions{.eps}
\fi

\newsiamremark{remark}{Remark}
\newsiamremark{hypothesis}{Hypothesis}
\crefname{hypothesis}{Hypothesis}{Hypotheses}
\newsiamthm{claim}{Claim}

\headers{PGGR}{Zhiwei Gao, George Karniadakis}

\title{Proposal-Guided Greedy Surrogate Refinement for PDE-Driven High-Dimensional Rare-Event Estimation}

\author{Zhiwei Gao\thanks{Applied Mathematics, Brown University, Providence 
  (\email{zhiwei\_gao@brown.edu}).}
\and George Karniadakis\thanks{Applied Mathematics, Brown University, Providence 
  (\email{george\_karniadakis@brown.edu}).}
}

\usepackage{amsopn}

\makeatletter
\newcommand*{\addFileDependency}[1]{
  \typeout{(#1)}
  \@addtofilelist{#1}
  \IfFileExists{#1}{}{\typeout{No file #1.}}
}
\makeatother

\newcommand*{\myexternaldocument}[1]{%
    \externaldocument{#1}%
    \addFileDependency{#1.tex}%
    \addFileDependency{#1.aux}%
}

\ifpdf
\hypersetup{
  pdftitle={Proposal-Guided Greedy Surrogate Refinement for PDE-Driven High-Dimensional  Rare-Event Estimation},
  pdfauthor={Zhiwei Gao, George Karniadakis}
}
\fi

\myexternaldocument{ex_supplement}

\begin{document}

\maketitle

\begin{abstract}
Accurate surrogate construction for PDE-driven high-dimensional rare-event
simulation is challenging when performance evaluations are expensive. Since a
globally accurate surrogate may require many high-fidelity evaluations, adaptive
importance sampling provides a natural localization tool: its evolving proposal
distribution progressively identifies the failure-relevant region. Motivated by
this observation, we propose a surrogate-assisted adaptive importance sampling
framework that refines the surrogate locally along the evolving proposal, rather
than over the entire input space. The surrogate combines an encoder with a neural
network, providing a low-dimensional latent representation for both prediction
and sample selection. At each adaptive iteration, candidates drawn from the
current proposal are selected by a greedy latent-space rule balancing proximity
to the estimated failure boundary and sample diversity. The selected samples are evaluated by the high-fidelity model and used to refine the surrogate, which then guides the subsequent cross-entropy-type adaptive
proposal update. We establish
one-step proposal stability bounds under local surrogate errors, together with
surrogate-induced misclassification and finite-sample estimation error bounds.
Numerical experiments on multimodal benchmarks and PDE-driven rare-event
problems up to 100 dimensions show that the proposed method achieves accuracy
comparable to true-model adaptive importance sampling while requiring
substantially fewer high-fidelity evaluations.
\end{abstract}



\section{Introduction}

Rare-event simulation \cite{bucklew2004introduction} aims to estimate the
probability that a system enters a failure state under uncertain inputs. Let
$\mathbf{u}\in\Omega\subseteq\mathbb{R}^d$ be a random input with density
$p(\mathbf{u})$, and let failure be described by a performance function
$g(\mathbf{u})$, so that failure occurs when $g(\mathbf{u})\le 0$. The failure set
and failure probability are
\begin{equation}
\label{eq:failure_probability}
\Omega_{\mathcal F}
=
\{\mathbf{u}\in\Omega:g(\mathbf{u})\le 0\},
\qquad
P_{\mathcal F}
=
\int_{\Omega}
\mathbb{I}_{\Omega_{\mathcal F}}(\mathbf{u})\,
p(\mathbf{u})\,d\mathbf{u}.
\end{equation}
Estimating $P_{\mathcal F}$ is difficult when the event is rare, the failure set
has complicated geometry, or each evaluation of $g$ requires an expensive
simulation or PDE solve.

Direct Monte Carlo simulation \cite{rubino2009rare} is robust but becomes
prohibitively expensive when $P_{\mathcal F}$ is small. Importance sampling (IS)
improves efficiency by sampling from a biased proposal distribution that places
more probability mass in the failure-relevant region. Adaptive IS methods,
including cross-entropy (CE)-type methods
\cite{kurtz2013cross, PAPAIOANNOU2019106564, gao2025safe}, further construct a
sequence of proposal distributions that gradually move from the nominal density
toward the rare-event region. Related rare-event strategies, such as subset
simulation \cite{au2001estimation, bect2017bayesian}, sequential Monte Carlo
\cite{cerou2012sequential, papaioannou2016sequential}, large-deviation-based
adaptive IS, multilevel rare-event estimation, and multilevel sequential
importance sampling
\cite{TongStadler2023,WagnerLatzPapaioannouUllmann2020,UllmannPapaioannou2015,LiuLuZhou2015},
also rely on intermediate distributions or levels to guide samples toward the
rare event.

Although adaptive IS can substantially reduce estimator variance, it may still
require many evaluations of the true performance function during the adaptive
stage. This motivates surrogate-assisted rare-event methods, where part of the
expensive model evaluations are replaced by an approximation of $g$. Classical
surrogates, including response surfaces
\cite{faravelli1989response, rajashekhar1993new}, polynomial chaos expansions
(PCE) \cite{he2020adaptive, ehre2022sequential}, and Gaussian processes (GP)
\cite{balesdent2013kriging, echard2011ak}, have been combined with Monte Carlo,
subset simulation, and importance sampling methods. Surrogate-based CE updates,
reduced basis approximations, and multifidelity importance sampling have also
been developed for failure-probability estimation
\cite{LiXiu2010,LiLiXiu2011,ChenQuarteroni2013,PeherstorferCuiMarzoukWillcox2016}.

However, constructing reliable surrogates for high-dimensional rare-event
problems remains challenging. Classical surrogate models are often effective in
low or moderate dimensions, but their sample requirements can grow rapidly with
dimension. Neural-network surrogates provide a natural alternative in high
dimensions because of their expressive power and their ability to learn nonlinear
representations from complex simulation data. Nevertheless, standard neural
networks do not directly provide reliable posterior uncertainty estimates, which
makes classical uncertainty-based active learning less straightforward. Existing
active-learning surrogate methods
\cite{li2021active, xiang2020active, ehre2022sequential, peijuan2017new}
often refine the surrogate near the limit-state surface using acquisition
functions tied to probabilistic surrogates such as Gaussian processes. Extending
such strategies to neural-network surrogates usually requires additional
machinery, such as ensembles \cite{tripathy2018deep}, Bayesian neural networks
\cite{mackay1992practical}, or dropout-based approximations
\cite{srivastava2014dropout}.

In this work, we exploit a simple observation: adaptive importance sampling
already produces a data-driven localization mechanism. As the proposal
distribution evolves, it identifies the region that contributes most to the
rare-event estimator. Therefore, a globally accurate surrogate over the entire
input space is often unnecessary; it is more important to refine the surrogate
locally under the evolving proposal distribution, where surrogate errors most
directly affect the proposal update and the final probability estimate
\cite{gao2024adaptive, yan2021stein}.

Based on this observation, we propose a proposal-guided greedy surrogate
refinement framework for high-dimensional rare-event estimation. The surrogate
combines an encoder with a neural network, so that the encoder provides a
low-dimensional latent representation for both prediction and sample selection.
At each adaptive iteration, candidates are drawn from the current proposal
distribution. A greedy selection rule then chooses samples that are close to the
estimated failure boundary while maintaining diversity in the latent space.
These selected samples are evaluated by the high-fidelity model and added to the
training set, thereby refining the surrogate in the proposal-induced region. The
refined surrogate then defines the next intermediate target in an improved cross-entropy-type proposal
update.

The main contributions of this work are summarized as follows:
\begin{itemize}
    \item We propose a proposal-guided surrogate refinement framework for
    high dimensional rare-event estimation, where the evolving adaptive-IS
    proposal determines where the surrogate should be refined.

    \item We develop a neural-network surrogate with a latent-space greedy
    selection rule that balances boundary proximity and sample diversity without
    relying on uncertainty-based acquisition functions.

    \item We instantiate the framework using an ICE-vMFNM proposal update
    \cite{PAPAIOANNOU2019106564} and analyze one-step proposal stability,
    surrogate-induced misclassification, and finite-sample estimation error.

    \item We validate the method on multimodal, high-dimensional, and PDE-driven
    examples, showing that accurate rare-event estimates can be obtained with
    substantially fewer high-fidelity evaluations.
\end{itemize}

The remainder of the paper is organized as follows.
Section~\ref{sec:adaptive_importance_sampling} reviews adaptive importance
sampling and the proposal-evolution viewpoint. Section~\ref{sec:surrogate_assisted_proposal_refinement}
presents the proposed surrogate-assisted adaptive proposal refinement framework.
Section~\ref{sec:error_analysis} provides the error analysis.
Section~\ref{sec:numerical_experiments} reports numerical experiments, and
Section~\ref{sec:conclusion} concludes the paper.

\section{Adaptive Importance Sampling}
\label{sec:adaptive_importance_sampling}

We briefly review adaptive importance sampling for rare-event simulation and
introduce the proposal-evolution viewpoint used in the proposed method. Given a
proposal density \(q(\mathbf u)\), the failure probability can be written as
\begin{equation}
    P_{\mathcal F}
    =
    \int_{\Omega}
    \mathbb{I}_{\Omega_{\mathcal F}}(\mathbf u)
    \frac{p(\mathbf u)}{q(\mathbf u)}
    q(\mathbf u)\,d\mathbf u
    =
    \mathbb{E}_{q}
    \left[
    \mathbb{I}_{\Omega_{\mathcal F}}(\mathbf u)
    w(\mathbf u)
    \right],
    \label{eq:is_identity}
\end{equation}
where \(w(\mathbf u)=p(\mathbf u)/q(\mathbf u)\) is the likelihood ratio.
Given independent samples \(\mathbf u_i\sim q\), the standard importance
sampling estimator is
\begin{equation}
    \widehat P_{\mathcal F}^{\mathrm{IS}}
    =
    \frac{1}{N}
    \sum_{i=1}^{N}
    \mathbb{I}_{\Omega_{\mathcal F}}(\mathbf u_i)
    w(\mathbf u_i).
    \label{eq:is_estimator}
\end{equation}
The zero-variance proposal is $q^\star(\mathbf u)
    =
    \mathbb{I}_{\Omega_{\mathcal F}}(\mathbf u)p(\mathbf u)/
    P_{\mathcal F}
    $, 
but it is unavailable in practice because it depends on both the unknown failure
probability and the exact failure set. Adaptive importance sampling therefore
constructs a sequence of proposals that gradually concentrate on the
failure-relevant region.

At adaptive stage \(t\), the proposal \(q_t\) is used as a tractable
approximation of an intermediate target density \(\pi_t\). The intermediate
target is defined as
\begin{equation}
    \pi_t(\mathbf u)
    =
    \frac{
    h_t(\mathbf u)p(\mathbf u)
    }{
    Z_t
    },
    \qquad
    Z_t
    =
    \int_{\Omega}
    h_t(\mathbf u)p(\mathbf u)\,d\mathbf u,
    \label{eq:intermediate_target}
\end{equation}
where \(h_t(\mathbf u)\ge 0\) is an importance function. Initially,
\(q_0=\pi_0=p\), corresponding to \(h_0\equiv 1\). As the adaptive procedure
proceeds, \(h_t\) is chosen to give more weight to samples closer to the failure
region. In the limiting case, if \(h_t\) approaches the failure indicator
\(\mathbb I_{\Omega_{\mathcal F}}\), then \(\pi_t\) approaches the
zero-variance proposal \(q^\star\).

Since direct sampling from \(\pi_t\) is generally unavailable, \(q_t\) is
obtained by fitting a parametric density to \(\pi_t\) using samples from the
previous proposal \(q_{t-1}\). Abstractly, for \(t\ge 1\), this update can be
written as
\begin{equation}
    q_t
    =
    \mathcal U
    \left(
    q_{t-1},
    \left\{
    \mathbf u_i^{(t-1)},g(\mathbf u_i^{(t-1)})
    \right\}_{i=1}^{N_c}
    \right),
    \qquad
    \mathbf u_i^{(t-1)}\sim q_{t-1},
    \label{eq:general_ais_update}
\end{equation}
where \(\mathcal U\) denotes a proposal-update operator and $N_{c}$ denotes the number of samples used for update. Different adaptive
importance sampling methods correspond to different choices of the importance
function \(h_t\), the proposal family, and the update operator
\(\mathcal U\) \cite{XIAN2024102393}.

After the adaptive stage, the final proposal \(q_{\mathrm f}\) can be used in
the true-model importance sampling estimator
\begin{equation}
    \widehat P_{\mathcal F}
    =
    \frac{1}{N}
    \sum_{i=1}^{N}
    \mathbb{I}_{\{g(\mathbf u_i)\le 0\}}
    \frac{p(\mathbf u_i)}{q_{\mathrm f}(\mathbf u_i)},
    \qquad
    \mathbf u_i\sim q_{\mathrm f}.
    \label{eq:final_is_estimator}
\end{equation}

When evaluating \(g\) is expensive, the adaptive stage may still be costly.
Moreover, in high-dimensional problems, building a globally accurate surrogate
over the entire input space is often inefficient. The proposal sequence provides
a natural localization mechanism: as \(q_t\) concentrates near the
failure-relevant region, the surrogate only needs to be refined in the region
explored by the current proposal. This motivates the proposal-guided surrogate
refinement strategy introduced next.

\section{Surrogate-Assisted Adaptive Proposal Refinement}
\label{sec:surrogate_assisted_proposal_refinement}

Motivated by the proposal-induced localization discussed above, we introduce the
proposed surrogate-assisted adaptive proposal refinement framework. The key idea
is to use a neural-network surrogate and refine it along the evolving proposal
distributions rather than over the whole input space.

\subsection{Adaptive surrogate construction}
\label{subsec:surrogate_construction}

Before the adaptive proposal refinement starts, we first construct an initial
surrogate model using samples generated from the original input distribution.
Specifically, we draw an initial set of samples from \(p\) and evaluate the true
performance function:
\[
    \mathcal{D}_0
    =
    \{(\mathbf{u}_j^{(0)},g(\mathbf{u}_j^{(0)}))\}_{j=1}^{M_0},
    \qquad
    \mathbf{u}_j^{(0)}\sim p .
\]
This initial dataset is used to pretrain a surrogate model
\(\widehat g_0=\mathcal{N}(\cdot;\theta_0)\), which provides a starting
approximation of the performance function and is reused in the subsequent
adaptive procedure.

As the proposal distribution evolves, the surrogate should be refined in the
region explored by the current proposal rather than over the entire input space. Given the
current proposal distribution \(q_t\), the local surrogate refinement is based on
a regularized training objective of the form
\begin{equation}
\label{loss_function}
    \mathcal L_t(\theta)
    =
    \left\|
    \mathcal{N}(\mathbf{x};\theta) - g(\mathbf{x})
    \right\|^{2}_{L_{2}(q_{t})}
    +
    \lambda \,\mathcal{R}(\theta),
\end{equation}
where \(L_2(q_t)\) denotes the Hilbert space equipped with the measure \(q_t\),
\(\mathcal{R}(\theta)\) is a regularization term, and \(\lambda>0\) is the
regularization weight used during surrogate training. Starting from the previous
surrogate parameters, the model is updated by a gradient-based optimizer \cite{kingma2014adam} to
obtain the refined surrogate for the next proposal update. In the numerical
implementation, \(\lambda\) is updated within the surrogate optimization process
by a smoothed gradient-balancing rule. Thus, \(\lambda\) is not treated as a
quantity indexed by the outer adaptive iteration.

To discretize the local objective in \eqref{loss_function}, samples should be
drawn from the current proposal distribution \(q_t\). However, evaluating the
true performance function for all proposal samples would be expensive. Therefore,
instead of labeling all samples from \(q_t\), we first generate a candidate pool
from \(q_t\) and then select a small subset of informative samples for
high-fidelity evaluation. These selected samples are added to the training set
and used to refine the surrogate in the proposal-induced region.

In high-dimensional settings, uncertainty estimation for standard neural-network
surrogates is generally not directly available, which makes classical
uncertainty-based active learning strategies difficult to apply. Moreover,
distance-based or diversity-based criteria in the original input space may
become less effective due to the curse of dimensionality. To address this issue,
we perform sample selection in a learned latent space.

Specifically, we decompose the neural-network surrogate into an encoder and a
feedforward prediction network: $\mathcal{N}(\mathbf{x};\theta)
    =
    \mathcal{F}_{\phi}
    \bigl(
    \mathcal{E}_{\psi}(\mathbf{x})
    \bigr)$, 
where \(\mathcal{E}_{\psi}\) maps the input \(\mathbf{x}\) to a latent
representation, \(\mathcal{F}_{\phi}\) maps the latent representation to the
predicted performance value, and \(\theta=(\psi,\phi)\). The encoder provides a low-dimensional representation used both for prediction
and for measuring sample diversity in the greedy selection step. The details of
the selection strategy are presented next.

\subsection{Greedy sample selection in latent space}
\label{greedy_selection}

To select informative samples for surrogate refinement, we use two criteria:
proximity to the estimated failure boundary and diversity among selected samples.
The proximity term targets points where surrogate sign errors most directly
affect failure classification, while the diversity term avoids redundant samples
and improves coverage of the proposal-induced important region.

At adaptive refinement iteration \(t\ge 1\), suppose that the current proposal
\(q_t\), the available surrogate \(\widehat g_{t-1}\), and the current training
dataset \(\mathcal D_{t-1}\) are given. Let the candidate pool generated from the
current proposal distribution \(q_t\) be denoted by
\begin{equation}
    \mathcal{C}_t
    =
    \{\mathbf{x}_i^{(t)}\}_{i=1}^{N_c},
    \qquad
    \mathbf{x}_i^{(t)} \sim q_t .
    \label{eq:candidate_pool}
\end{equation}
The surrogate used for selection is written as $\widehat g_{t-1}(\mathbf{x})
    =
    \mathcal{F}_{\phi_{t-1}}\bigl(\mathcal{E}_{\psi_{t-1}}(\mathbf{x})\bigr)$, 
where $\mathcal{E}_{\psi_{t-1}}$ is the encoder and
$\mathcal{F}_{\phi_{t-1}}$ is the prediction network. The encoder maps a
high-dimensional input $\mathbf{x}\in\mathbb{R}^d$ to a low-dimensional latent
representation $\mathbf{z}
    =
    \mathcal{E}_{\psi_{t-1}}(\mathbf{x})$. 
Suppose the current training dataset is
 $\mathcal{D}_{t-1}
    =
    \{(\mathbf{x}_j,g(\mathbf{x}_j))\}_{j=1}^{M_{t-1}}$. For each candidate $\mathbf{x}\in\mathcal{C}_t$, the proximity to the estimated
failure boundary is measured by $|\widehat g_{t-1}(\mathbf{x})|$. A smaller value
indicates that the candidate is closer to the current approximation of the
failure boundary. To measure diversity, we compute distances in the latent space.
Given a reference set $\mathcal{R}\subset \mathbb{R}^d$, define
\begin{equation}
    d_t(\mathbf{x},\mathcal{R})
    =
    \min_{\mathbf{y}\in \mathcal{R}}
    \left\|
        \mathcal{E}_{\psi_{t-1}}(\mathbf{x})
        -
        \mathcal{E}_{\psi_{t-1}}(\mathbf{y})
    \right\|_2 .
    \label{eq:latent_distance}
\end{equation}
During the greedy selection process, the reference set is updated as $\mathcal{R}_t^{(k)}
    =
    \{\mathbf{x}_j : (\mathbf{x}_j,g(\mathbf{x}_j))\in \mathcal{D}_{t-1}\}
    \cup
    \mathcal{A}_t^{(k)}$, 
where $\mathcal{A}_t^{(k)}$ denotes the set of samples already selected after
$k$ greedy steps. Thus, the diversity term accounts for both the existing
training samples and the newly selected samples.

Before computing the greedy score, both the boundary-proximity term
\(|\widehat g_{t-1}(\mathbf{x})|\) and the latent-space distance
\(d_t(\mathbf{x},\mathcal R_t^{(k)})\) are normalized to \([0,1]\) over the
remaining candidate pool \(\mathcal C_t\setminus\mathcal A_t^{(k)}\). We then
define the greedy score as
\begin{equation}
    S_t^{(k)}(\mathbf{x})
    =
    -
    \widetilde{|\widehat g_{t-1}|}(\mathbf{x})
    +
    \beta
    \widetilde d_t(\mathbf{x},\mathcal R_t^{(k)}),
    \label{eq:greedy_score}
\end{equation}
where \(\widetilde{|\widehat g_{t-1}|}\) and \(\widetilde d_t\) denote the
normalized boundary-proximity and diversity terms, respectively. The parameter
\(\beta\ge 0\) balances boundary proximity and latent-space diversity.
The first term favors candidates close to the estimated boundary
\(\widehat g_{t-1}(\mathbf{x})=0\), while the second promotes latent-space
diversity and avoids redundant enrichment
\cite{echard2011ak, peijuan2017new, xiang2020active}.

Starting from $\mathcal{A}_t^{(0)}=\emptyset$, we iteratively select
\[
    \mathbf{x}^{\star}
    =
    \arg\max_{\mathbf{x}\in \mathcal{C}_t\setminus \mathcal{A}_t^{(k)}}
    S_t^{(k)}(\mathbf{x}),
\]
and update $\mathcal{A}_t^{(k+1)}
    =
    \mathcal{A}_t^{(k)}\cup \{\mathbf{x}^{\star}\}$.
This procedure is repeated until $|\mathcal{A}_t|=m_t$, where $m_t$ is the
number of high-fidelity evaluations added at iteration $t$.

The selected samples are then evaluated by the true performance function $g$,
and the dataset is updated as
\begin{equation}
    \mathcal{D}_{t}
    =
    \mathcal{D}_{t-1}
    \cup
    \{(\mathbf{x},g(\mathbf{x})):\mathbf{x}\in \mathcal{A}_t\}.
    \label{eq:dataset_update}
\end{equation}
The updated dataset $\mathcal{D}_{t}$ is subsequently used to refine the
surrogate, yielding $\widehat g_t$, which is then used in the next ICE-vMFNM
proposal update. The greedy rule for constructing $\mathcal A_t$ is summarized
in Algorithm~\ref{alg:greedy_latent_selection}.

\begin{algorithm}[t]
\caption{Greedy sample selection in latent space}
\label{alg:greedy_latent_selection}
\begin{algorithmic}[1]
\Require Candidate pool $\mathcal{C}_t$, available surrogate
$\widehat g_{t-1}=\mathcal{F}_{\phi_{t-1}}\circ\mathcal{E}_{\psi_{t-1}}$,
dataset $\mathcal{D}_{t-1}$, batch size $m_t$, and weight $\beta\ge 0$
\Ensure Selected sample set $\mathcal{A}_t$

\State Initialize the selected set $\mathcal{A}_t^{(0)} \gets \emptyset$.

\For{$k=0,1,\ldots,m_t-1$}

    \State Construct the current reference set $\mathcal{R}_t^{(k)}
        =
        \{\mathbf{x}_j:(\mathbf{x}_j,g(\mathbf{x}_j))\in\mathcal{D}_{t-1}\}
        \cup
        \mathcal{A}_t^{(k)}$. 

    \State For each
    $\mathbf{x}\in\mathcal{C}_t\setminus\mathcal{A}_t^{(k)}$,
    compute its latent distance
    $d_t(\mathbf{x},\mathcal{R}_t^{(k)})$ by \eqref{eq:latent_distance}.

    \State Compute the greedy score
    $S_t^{(k)}(\mathbf{x})$ by \eqref{eq:greedy_score}.

    \State Select the candidate with the largest greedy score:
    \[
        \mathbf{x}^{\star}
        =
        \arg\max_{\mathbf{x}\in\mathcal{C}_t\setminus\mathcal{A}_t^{(k)}}
        S_t^{(k)}(\mathbf{x}) .
    \]

    \State Update the selected set: $\mathcal{A}_t^{(k+1)}
        \gets
        \mathcal{A}_t^{(k)} \cup \{\mathbf{x}^{\star}\}$.

\EndFor

\State Set $\mathcal{A}_t \gets \mathcal{A}_t^{(m_t)}$.

\State \Return $\mathcal{A}_t$.

\end{algorithmic}
\end{algorithm}
\subsection{Improved cross-entropy implementation of the proposed framework}
\label{subsec:surrogate_ce_update}

We now describe how the proposal-guided surrogate refinement is combined with an
improved cross-entropy-type proposal update. In this subsection, the improved
cross-entropy (ICE) method is used as a concrete realization of the general
update operator \(\mathcal U\) in \eqref{eq:general_ais_update}. 

We use a parametric mixture proposal distribution. For \(t\ge 1\), the proposal
is chosen from a \(K\)-component vMFNM mixture
\cite{PAPAIOANNOU2019106564} family:
\[
    q_t(\mathbf{u})
    =
    q(\mathbf{u};\boldsymbol{\eta}_t)
    =
    \sum_{k=1}^{K}
    w_{k,t}
    q_{\mathrm{vMFNM}}(\mathbf{u};
    \boldsymbol{\mu}_{k,t},
    \kappa_{k,t},
    m_{k,t},
    \Omega_{k,t}),
\]
where $\boldsymbol{\eta}_t
    =
    \left\{
    w_{k,t},
    \boldsymbol{\mu}_{k,t},
    \kappa_{k,t},
    m_{k,t},
    \Omega_{k,t}
    \right\}_{k=1}^{K}$ 
denotes the proposal parameters, with \(w_{k,t}\ge 0\) and
\(\sum_{k=1}^{K}w_{k,t}=1\). The mixture structure allows the proposal to
represent multiple important regions, while each vMFNM component captures
directional and radial concentration in high-dimensional spaces.

The algorithm starts from an initial dataset \(\mathcal D_0\) and an initial
surrogate \(\widehat g_0\) trained under the nominal density \(p\). We set
\(q_0=p\). For each adaptive stage \(t\ge 0\), we generate a candidate pool from
the current proposal:
\[
    \mathcal C_t
    =
    \{\mathbf u_i^{(t)}\}_{i=1}^{N_c},
    \qquad
    \mathbf u_i^{(t)}\sim q_t .
\]
At the initialization stage \(t=0\), the pool \(\mathcal C_0\), together with
\(\widehat g_0\), is used to define the first intermediate target \(\pi_1\) and
to fit the first adaptive proposal \(q_1\).

For \(t\ge 1\), suppose that the current proposal \(q_t\), the available
surrogate \(\widehat g_{t-1}\), and the current training dataset
\(\mathcal D_{t-1}\) are given. The candidate pool \(\mathcal C_t\) is first used
for greedy sample selection based on \(\widehat g_{t-1}\), producing a selected
set \(\mathcal A_t\subset \mathcal C_t\). The true performance function is
evaluated only on the selected samples in \(\mathcal A_t\), and the dataset is
updated according to Eq.~\eqref{eq:dataset_update}. The surrogate is then refined
using the updated dataset \(\mathcal D_t\), yielding \(\widehat g_t\).

Specifically, given the refined surrogate \(\widehat g_t\), we define the soft
importance function used in the ICE strategy
\cite{PAPAIOANNOU2019106564,zhang2025ice} by
\begin{equation}
\label{soft_indicator}
    h_{t+1}(\mathbf u)
    =
    \Phi
    \left(
    -\frac{\widehat g_t(\mathbf u)}{\sigma_{t+1}}
    \right),
\end{equation}
where \(\Phi\) is the standard normal cumulative distribution function and
\(\sigma_{t+1}>0\) is a smoothing parameter. Following the ICE strategy,
\(\sigma_{t+1}\) is selected by controlling the coefficient of variation of the
importance weights. Given \(\mathbf u_i^{(t)}\sim q_t\), we choose
\[
    \sigma_{t+1}
    =
    \arg\min_{\sigma\in(0,\sigma_t)}
    \left(
    \delta_{W_{t+1}}(\sigma)-\delta_{\mathrm{target}}
    \right)^2,
\]
where
\[
    W_i^{(t+1)}(\sigma)
    =
    \Phi
    \left(
    -\frac{\widehat g_t(\mathbf u_i^{(t)})}{\sigma}
    \right)
    \frac{p(\mathbf u_i^{(t)})}{q_t(\mathbf u_i^{(t)})},
\]
and \(\delta_{W_{t+1}}(\sigma)\) denotes the sample coefficient of variation of
\(\{W_i^{(t+1)}(\sigma)\}_{i=1}^{N_c}\). Here \(\sigma_0>0\) is prescribed as an
initial smoothing upper bound. With \(\sigma_{t+1}\) determined, the corresponding
intermediate target \(\pi_{t+1}\) is defined according to
\eqref{eq:intermediate_target}.

The stopping criterion follows the ICE rule. Since the soft indicator should
eventually approximate the failure indicator, we define the diagnostic weights
\begin{equation}
\label{stop_criteria}
    W_{t+1,i}^{\ast}
    =
    \frac{
    \mathbb{I}_{\{\widehat g_t(\mathbf u_i^{(t)})\le 0\}}
    }{
    h_{t+1}(\mathbf u_i^{(t)})
    },
    \qquad
    \mathbf u_i^{(t)}\sim q_t .
\end{equation}
If the sample coefficient of variation
\(\delta_{W_{t+1}^{\ast}}\) is below a prescribed stopping tolerance
\(\delta_{\mathrm{stop}}\), then the current proposal is regarded as sufficiently
close to the failure-relevant target. In this case, we set $q_{\mathrm f}=q_t,
    \widehat g_{\mathrm f}=\widehat g_t$, 
and terminate the adaptive stage.

In the numerical implementation, the denominator in \eqref{stop_criteria} is
bounded away from zero by a small positive constant, i.e., $h_{t+1}(\mathbf u_i^{(t)})
    \leftarrow
    \max\{h_{t+1}(\mathbf u_i^{(t)}),\varepsilon_h\}$, 
and the resulting weights are normalized after evaluation. This safeguard is used
only for numerical stability and is not part of the idealized theoretical
formulation.

If the stopping criterion is not satisfied, the next proposal \(q_{t+1}\) is
obtained by fitting the vMFNM mixture distribution to \(\pi_{t+1}\). Following
the cross-entropy principle, this update is formulated as the KL projection
\begin{equation}
    \boldsymbol{\eta}_{t+1}
    =
    \arg\min_{\boldsymbol{\eta}}
    \operatorname{KL}
    \left(
    \pi_{t+1}
    \,\|\,
    q(\cdot;\boldsymbol{\eta})
    \right)
    =
    \arg\max_{\boldsymbol{\eta}}
    \mathbb{E}_{\pi_{t+1}}
    \left[
    \log q(\mathbf u;\boldsymbol{\eta})
    \right].
    \label{eq:ce_kl_vMFNM}
\end{equation}
Using the same candidate pool \(\mathcal C_t\), this expectation is approximated
by importance reweighting, which gives the weighted maximum-likelihood update
\begin{equation}
    \boldsymbol{\eta}_{t+1}
    =
    \arg\max_{\boldsymbol{\eta}}
    \sum_{i=1}^{N_c}
    \bar W_i^{(t+1)}
    \log q(\mathbf u_i^{(t)};\boldsymbol{\eta}),
    \qquad
    \mathbf u_i^{(t)}\sim q_t,
    \label{eq:surrogate_ce_update}
\end{equation}
where
\[
    \bar W_i^{(t+1)}
    =
    \frac{
    W_i^{(t+1)}
    }{
    \sum_{j=1}^{N_c}W_j^{(t+1)}
    },
    \qquad
    W_i^{(t+1)}
    =
    h_{t+1}(\mathbf u_i^{(t)})
    \frac{
    p(\mathbf u_i^{(t)})
    }{
    q_t(\mathbf u_i^{(t)})
    }.
\]
The parameters \(\boldsymbol{\eta}_{t+1}\) are obtained by applying a weighted EM
procedure~\cite{PAPAIOANNOU2019106564}. Therefore,
\[
    q_{t+1}
    =
    \mathcal{U}_{\mathrm{ICE\text{-}vMFNM}}
    \left(
    q_t,
    \widehat g_t
    \right)
    =
    q(\cdot;\boldsymbol{\eta}_{t+1}),
\]
where \(\boldsymbol{\eta}_{t+1}\) is computed from
\eqref{eq:surrogate_ce_update}. The proposal \(q_{t+1}\) is then used in the next
adaptive refinement iteration. If the maximum number of adaptive iterations is
reached before the stopping criterion is satisfied, the most recently obtained
proposal and surrogate are used as \(q_{\mathrm f}\) and \(\widehat g_{\mathrm f}\).

The whole procedure is illustrated in Fig.~\ref{fig:pggr} and summarized in
Algorithm~\ref{alg:surrogate_ce_vmfnm}.

\begin{figure}[htbp]
    \centering
    \includegraphics[width=0.75\linewidth]{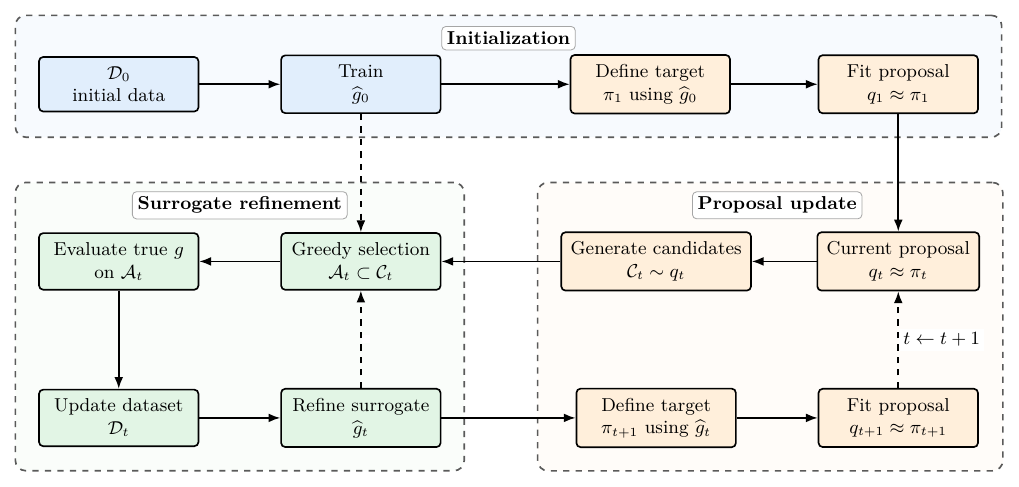}
    \caption{Schematic of the proposal-guided greedy surrogate refinement procedure.
The proposal \(q_t\) is fitted to the intermediate target \(\pi_t\). Candidate
samples are generated from \(q_t\), and the available surrogate
\(\widehat g_{t-1}\) is used to select informative samples for high-fidelity
evaluation. The updated dataset \(\mathcal D_t\) yields the refined surrogate
\(\widehat g_t\), which defines the soft importance function \(h_{t+1}\) and the
corresponding stopping diagnostic. If the stopping criterion is not satisfied,
the next target \(\pi_{t+1}\) is fitted to obtain the next proposal \(q_{t+1}\).}
    \label{fig:pggr}
\end{figure}

\begin{algorithm}[t]
\caption{PGGR-ICE-vMFNM adaptive importance sampling}
\label{alg:surrogate_ce_vmfnm}
\begin{algorithmic}[1]
\Require Initial size \(M_0\), candidate size \(N_c\), batch size \(m_t\),
target CoV \(\delta_{\mathrm{target}}\), stopping tolerance
\(\delta_{\mathrm{stop}}\), and maximum iteration number \(T_{\max}\)
\Ensure Final proposal \(q_{\mathrm f}\), final surrogate \(\widehat g_{\mathrm f}\),
and estimate \(\widehat P_{\mathcal F}^{\mathrm{surr}}\)

\State Generate \(\mathcal D_0=\{(\mathbf u_j^{(0)},g(\mathbf u_j^{(0)}))\}_{j=1}^{M_0}\), with \(\mathbf u_j^{(0)}\sim p\).
\State Train \(\widehat g_0\) using \(\mathcal D_0\) and set \(q_0=p\).
\State Generate \(\mathcal C_0=\{\mathbf u_i^{(0)}\}_{i=1}^{N_c}\), with \(\mathbf u_i^{(0)}\sim q_0\).
\State Use \(\widehat g_0\) and \(\mathcal C_0\) to choose \(\sigma_1\), define \(h_1\), and fit \(q_1\) by \eqref{eq:surrogate_ce_update}.

\For{\(t=1,\ldots,T_{\max}\)}
    \State Generate \(\mathcal C_t=\{\mathbf u_i^{(t)}\}_{i=1}^{N_c}\), with \(\mathbf u_i^{(t)}\sim q_t\).
    \State Select \(\mathcal A_t\subset\mathcal C_t\) using \(\widehat g_{t-1}\) and Algorithm~\ref{alg:greedy_latent_selection}.
    \State Evaluate \(g(\mathbf u)\) for \(\mathbf u\in\mathcal A_t\) and update \(\mathcal D_t\) by \eqref{eq:dataset_update}.
    \State Refine the surrogate using \(\mathcal D_t\) to obtain \(\widehat g_t\).
    \State Use \(\mathcal C_t\) and \(\widehat g_t\) to choose \(\sigma_{t+1}\), define \(h_{t+1}\), and compute \(\delta_{W_{t+1}^{\ast}}\) by \eqref{stop_criteria}.

    \If{\(\delta_{W_{t+1}^{\ast}}\le \delta_{\mathrm{stop}}\)}
        \State Set \(q_{\mathrm f}=q_t\), \(\widehat g_{\mathrm f}=\widehat g_t\), and terminate the adaptive stage.
        \State \textbf{break}
    \Else
        \State Fit \(q_{t+1}\) by \eqref{eq:surrogate_ce_update}.
    \EndIf
\EndFor

\If{the stopping criterion is not met before \(T_{\max}\)}
    \State Set \(q_{\mathrm f}=q_{T_{\max}+1}\) and \(\widehat g_{\mathrm f}=\widehat g_{T_{\max}}\).
\EndIf

\State Generate final samples \(\mathbf u_i\sim q_{\mathrm f}\) and compute
\(\widehat P_{\mathcal F}^{\mathrm{surr}}\) by \eqref{eq:surrogate_final_estimator}.
\end{algorithmic}
\end{algorithm}

\paragraph{Surrogate-based final estimator}
After the adaptive proposal update stops, let \(q_{\mathrm f}\) and
\(\widehat g_{\mathrm f}\) denote the resulting final proposal distribution and
final refined surrogate, respectively. To reduce the number of high-fidelity
evaluations, the proposed method uses the surrogate-induced failure indicator in
the final probability estimator:
\begin{equation}
    \widehat P_{\mathcal F}^{\mathrm{surr}}
    =
    \frac{1}{N}
    \sum_{i=1}^{N}
    \mathbb{I}_{\{\widehat g_{\mathrm f}(\mathbf u_i)\le 0\}}
    \frac{p(\mathbf u_i)}{q_{\mathrm f}(\mathbf u_i)},
    \qquad
    \mathbf u_i\sim q_{\mathrm f}.
    \label{eq:surrogate_final_estimator}
\end{equation}
Here, the true performance function \(g\) is evaluated only for the initial
training samples and the greedily selected refinement samples. Therefore, the
high-fidelity cost consists of the initial training evaluations and the selected
samples added during the adaptive refinement stage; the precise reporting
convention is specified in Section~\ref{sec:numerical_experiments}.

\section{Theoretical Analysis}
\label{sec:error_analysis}

In this section, we analyze how the surrogate affects the proposed
surrogate-assisted adaptive importance sampling framework. The surrogate enters
the method in two distinct ways. First, during the adaptive stage, it is used to
construct the soft importance function and therefore influences the proposal
update. Second, after the adaptive stage, the final estimator uses the
surrogate-induced failure indicator, which may introduce a misclassification
bias.

At adaptive stage \(t\), the relevant quantity is the local surrogate error
\(\|g-\widehat g_t\|_{L^2(q_t)}\) under the current proposal distribution. This
error affects the stopping diagnostic and, when a new proposal is fitted, the
surrogate-driven ICE update. We then state a misclassification bound and a
finite-sample error decomposition for the final surrogate-based importance
sampling estimator. The proofs are deferred to Appendix~\ref{app:proofs}.

\subsection{One-step stability of the surrogate-driven ICE proposal update}
\label{subsec:one_step_stability}

At adaptive stage \(t\ge 0\), suppose that \(q_t\) is the current proposal and
that \(\widehat g_t\) is the surrogate used to fit the next intermediate target.
For the smoothed ICE update, define the true and surrogate-based soft importance
functions by
\begin{equation}
    h_{t+1}^\star(\mathbf u)
    =
    \Phi\left(
    -\frac{g(\mathbf u)}{\sigma_{t+1}}
    \right),
    \qquad
    \widehat h_{t+1}(\mathbf u)
    =
    \Phi\left(
    -\frac{\widehat g_t(\mathbf u)}{\sigma_{t+1}}
    \right),
    \label{eq:true_surrogate_soft_indicators}
\end{equation}
where \(\Phi\) is the standard normal cumulative distribution function and
\(\sigma_{t+1}>0\) is the smoothing parameter. The corresponding intermediate
target densities are
\begin{equation}
    \pi_{t+1}^\star(\mathbf u)
    =
    \frac{
    h_{t+1}^\star(\mathbf u)p(\mathbf u)
    }{
    Z_{t+1}^\star
    },
    \qquad
    \widehat \pi_{t+1}(\mathbf u)
    =
    \frac{
    \widehat h_{t+1}(\mathbf u)p(\mathbf u)
    }{
    \widehat Z_{t+1}
    },
    \label{eq:true_surrogate_intermediate_targets}
\end{equation}
where
\[
    Z_{t+1}^\star
    =
    \int h_{t+1}^\star(\mathbf u)p(\mathbf u)\,d\mathbf u,
    \qquad
    \widehat Z_{t+1}
    =
    \int \widehat h_{t+1}(\mathbf u)p(\mathbf u)\,d\mathbf u .
\]

Let \(\Pi_{\mathcal Q}\) denote the population cross-entropy projection onto the
proposal family \(\mathcal Q\),
\[
    \Pi_{\mathcal Q}[\pi]
    =
    \arg\min_{q\in\mathcal Q}
    \operatorname{KL}(\pi\|q).
\]
The proposal that would be obtained using the true performance function is
\(q_{t+1}^{\star}
    =
    \Pi_{\mathcal Q}[\pi_{t+1}^\star]\),
whereas the surrogate-driven update gives
\(\widehat q_{t+1}
    =
    \Pi_{\mathcal Q}[\widehat \pi_{t+1}]\). We use the following local assumptions.

\begin{assump}[Likelihood-ratio control]
\label{ass:lr_control}
There exists \(C_w^{(2)}<\infty\) such that
\begin{equation}
    \left\|
    \frac{p}{q_t}
    \right\|_{L^2(q_t)}^2
    =
    \int
    \frac{p(\mathbf u)^2}{q_t(\mathbf u)}
    d\mathbf u
    \le C_w^{(2)} .
    \label{eq:lr_control}
\end{equation}
\end{assump}

\begin{assump}[Mass lower bound]
\label{ass:soft_mass_lower_bound}
There exists \(z_{t+1}>0\) such that
\begin{equation}
    Z_{t+1}^\star
    =
    \int h_{t+1}^\star(\mathbf u)p(\mathbf u)\,d\mathbf u
    \ge z_{t+1} .
    \label{eq:soft_mass_lower_bound}
\end{equation}
\end{assump}

\begin{assump}[Local stability of the cross-entropy projection]
\label{ass:projection_stability}
There exists \(L_{\mathcal Q}<\infty\) such that, for probability measures
\(\mu\) and \(\nu\) in a neighborhood of the intermediate targets,
\begin{equation}
    \operatorname{TV}
    \left(
    \Pi_{\mathcal Q}[\mu],
    \Pi_{\mathcal Q}[\nu]
    \right)
    \le
    L_{\mathcal Q}
    \operatorname{TV}(\mu,\nu).
    \label{eq:projection_stability}
\end{equation}
\end{assump}

\begin{theo}[Conditional one-step proposal stability]
\label{thm:one_step_proposal_stability}
Under Assumptions~\ref{ass:lr_control}--\ref{ass:projection_stability}, the
surrogate-driven ICE update satisfies
\begin{equation}
    \operatorname{TV}
    \left(
    q_{t+1}^{\star},
    \widehat q_{t+1}
    \right)
    \le
    \frac{
    C_{\mathrm{stab}}
    }{
    z_{t+1}\sigma_{t+1}
    }
    \left\|
    g-\widehat g_t
    \right\|_{L^2(q_t)},
    \label{eq:one_step_proposal_stability}
\end{equation}
where \(
    C_{\mathrm{stab}}
    =
    2L_{\mathcal Q}\sqrt{C_w^{(2)}}/
    \sqrt{2\pi}
    \). 
\end{theo}

Theorem~\ref{thm:one_step_proposal_stability} shows that the perturbation of one
ICE proposal update is controlled by the local surrogate error
\(\|g-\widehat g_t\|_{L^2(q_t)}\), rather than by a global error under the
nominal distribution. This supports the proposal-guided refinement strategy:
samples from \(q_t\) are used to refine the surrogate, and the refined surrogate
\(\widehat g_t\) is then used to form the next intermediate target
\(\pi_{t+1}\) and proposal \(q_{t+1}\). The factor \(1/\sigma_{t+1}\) further
indicates that later ICE stages are more sensitive to surrogate error as the
soft indicator becomes sharper.

\begin{remark}
Theorem~\ref{thm:one_step_proposal_stability} is a conditional one-step
population stability result. Under a uniform local stability condition, the
accumulated proposal drift over \(T_{\mathrm{ad}}\) adaptive stages would depend
on quantities of the form \(
\sum_{t=0}^{T_{\mathrm{ad}}-1}
\frac{1}{z_{t+1}\sigma_{t+1}}
\|g-\widehat g_t\|_{L^2(q_t)}\). 
A full multi-step finite-sample analysis would additionally require
concentration bounds for the weighted mixture fitting step and stability of the
data-dependent surrogate refinement procedure, which is beyond the scope of this
work.
\end{remark}

\subsection{Surrogate-induced misclassification and estimator error}
\label{subsec:surrogate_estimator_error}

We now analyze the surrogate-based importance sampling estimator for a fixed
proposal \(q_t\) and surrogate \(\widehat g_t\). The result should be understood
as a fixed-proposal error decomposition, rather than a convergence guarantee for
an arbitrary adaptive proposal. The proposal quality enters through the
finite-sample term, which depends on the \(\chi^2\) divergence between the
surrogate-induced zero-variance density and \(q_t\); this quantity may be large,
or even infinite, if \(q_t\) does not adequately cover the relevant failure-biased
region. When the adaptive procedure stops, the result is applied with
\(t=t_{\mathrm f}\), \(q_t=q_{\mathrm f}\), and
\(\widehat g_t=\widehat g_{\mathrm f}\), thereby decomposing the final
surrogate-based estimation error into surrogate-induced misclassification bias
and finite-sample sampling error.

At iteration \(t\), define the true and surrogate-induced failure indicators by $I(\mathbf u)
    =
    \mathbb I_{\{g(\mathbf u)\le 0\}},
    \widehat I_t(\mathbf u)
    =
    \mathbb I_{\{\widehat g_t(\mathbf u)\le 0\}}.$
The corresponding misclassification set is
\[
    \mathcal M_t
    =
    \{\mathbf u:I(\mathbf u)\ne \widehat I_t(\mathbf u)\}.
\]
Since failure is determined by the sign of the performance function, a
misclassification can occur only when the surrogate error changes the sign of
\(g\). Hence,
\[
    \mathcal M_t
    \subseteq
    \left\{
    \mathbf u:
    |g(\mathbf u)|
    \le
    |g(\mathbf u)-\widehat g_t(\mathbf u)|
    \right\}.
\]

We impose the following standard margin condition near the failure boundary under
the current proposal distribution.

\begin{assump}[Margin condition under the proposal]
\label{ass:margin_condition}
There exist constants \(C_m>0\), \(\kappa>0\), and \(\tau_0>0\) such that, for all
\(0<\tau\le \tau_0\),
\begin{equation}
    q_t\bigl(|g(\mathbf u)|\le \tau\bigr)
    \le
    C_m\tau^\kappa .
    \label{eq:margin_condition}
\end{equation}
\end{assump}

This condition controls the amount of proposal probability mass near the failure
boundary. Under this condition, the surrogate-induced classification error can
be bounded by the local surrogate approximation error under \(q_t\).

\begin{lem}[Surrogate-induced event misclassification]
\label{lem:surrogate_misclassification}
Suppose Assumption~\ref{ass:margin_condition} holds. Then there exists a
constant \(C>0\), depending only on \(C_m\) and \(\kappa\), such that
\begin{equation}
    q_t(\mathcal M_t)
    \le
    C
    \|g-\widehat g_t\|_{L^2(q_t)}^{\frac{2\kappa}{\kappa+2}} .
    \label{eq:misclassification_bound}
\end{equation}
\end{lem}

Lemma~\ref{lem:surrogate_misclassification} shows that the relevant
classification error is controlled by the local surrogate error under the
proposal distribution \(q_t\), rather than by a global surrogate error over the
nominal distribution.

We next connect this classification error to the finite-sample importance
sampling estimator. For a fixed adaptive iteration \(t\), define
\[
    \widehat P_{N,t}^{\mathrm{surr}}
    =
    \frac{1}{N}
    \sum_{i=1}^{N}
    \widehat I_t(\mathbf u_i)
    \frac{p(\mathbf u_i)}{q_t(\mathbf u_i)},
    \qquad
    \mathbf u_i\sim q_t .
\]
This estimator targets the surrogate-induced failure probability
\[
    P_{\widehat{\mathcal F},t}
    =
    \mathbb{E}_{q_t}
    \left[
    \widehat I_t(\mathbf u)
    \frac{p(\mathbf u)}{q_t(\mathbf u)}
    \right].
\]
Define the corresponding surrogate-induced zero-variance density by $\widehat q_t^\star(\mathbf u)
    =
    \frac{
    \widehat I_t(\mathbf u)p(\mathbf u)
    }{
    P_{\widehat{\mathcal F},t}
    }$.

\begin{theo}[Finite-sample surrogate error decomposition]
\label{thm:finite_sample_surrogate_error}
Suppose Assumption~\ref{ass:margin_condition} holds. Assume further that the
likelihood ratio is bounded on the relevant region, namely
\begin{equation}
    \left\|
    \frac{p}{q_t}
    \right\|_{\infty}
    \le
    C_w^{(\infty)} .
    \label{eq:bounded_likelihood_ratio}
\end{equation}
Assume \(P_{\widehat{\mathcal F},t}>0\) and
\(\chi^2(\widehat q_t^\star\|q_t)<\infty\). Then there exists a constant
\(C>0\), depending on \(C_m\), \(\kappa\), and \(C_w^{(\infty)}\), such that
\begin{equation}
    \mathbb{E}
    \left[
    \left|
    \widehat P_{N,t}^{\mathrm{surr}}
    -
    P_{\mathcal F}
    \right|
    \right]
    \le
    C
    \|g-\widehat g_t\|_{L^2(q_t)}^{\frac{2\kappa}{\kappa+2}}
    +
    \frac{
    P_{\widehat{\mathcal F},t}
    }{
    \sqrt{N}
    }
    \sqrt{
    \chi^2(\widehat q_t^\star\|q_t)
    } .
    \label{eq:finite_sample_chi_square_bound}
\end{equation}
Equivalently, since
\[
    P_{\widehat{\mathcal F},t}
    \le
    P_{\mathcal F}
    +
    \left|P_{\widehat{\mathcal F},t}-P_{\mathcal F}\right|,
\]
and the surrogate-induced probability error is bounded by the first term in
\eqref{eq:finite_sample_chi_square_bound}, the sampling term can be further
controlled by
\begin{equation}
    \frac{
    P_{\widehat{\mathcal F},t}
    }{
    \sqrt{N}
    }
    \sqrt{\chi^2(\widehat q_t^\star\|q_t)}
    \le
    \frac{
    P_{\mathcal F}
    +
    C
    \|g-\widehat g_t\|_{L^2(q_t)}^{\frac{2\kappa}{\kappa+2}}
    }{
    \sqrt{N}
    }
    \sqrt{\chi^2(\widehat q_t^\star\|q_t)} .
    \label{eq:finite_sample_non_circular_bound}
\end{equation}
This form makes explicit that the bound depends on the true failure probability
and the surrogate-induced probability error, rather than treating
\(P_{\widehat{\mathcal F},t}\) as an independent quantity.
\end{theo}

Theorem~\ref{thm:finite_sample_surrogate_error} decomposes the final
surrogate-based estimation error into two parts. The first term is the
surrogate-induced bias caused by event misclassification. The second term is the
finite-sample sampling error under the proposal \(q_t\), expressed through the
\(\chi^2\) divergence between the surrogate-induced zero-variance density
\(\widehat q_t^\star\) and the proposal \(q_t\). Thus, accurate estimation
requires both a locally accurate surrogate near the failure boundary and a
proposal distribution that covers the surrogate-induced failure region.

\section{Numerical experiments}
\label{sec:numerical_experiments}

In this section, we evaluate the proposed surrogate-assisted adaptive importance
sampling framework on several rare-event simulation problems. The experiments
assess both estimation accuracy and high-fidelity evaluation cost, with emphasis
on high-dimensional settings where global surrogate construction is difficult.
\subsection{Experimental setup}
\label{subsec:experimental_setup}
For all numerical examples, we compare the proposed PGGR-ICE-vMFNM method
with the following baseline methods:
\begin{itemize}
    \item \textbf{Crude Monte Carlo (CMC):} used to provide the reference estimate
    $P_{\mathcal F}^{\rm ref}$ when such a reference is computationally feasible.

   \item \textbf{ICE-vMFNM:} the true-model adaptive importance sampling
baseline, which uses the same proposal family and adaptive proposal update as
the proposed method, but evaluates the true performance function \(g(\mathbf u)\)
throughout the adaptive stage.

\item \textbf{Random-ICE-vMFNM:} a random-refinement baseline used in selected
examples. It uses the same surrogate architecture, proposal update, and
comparable high-fidelity budget as PGGR-ICE-vMFNM, but replaces the greedy
refinement rule with uniform random selection from the candidate pool.
\end{itemize}

Here PGGR stands for Proposal-Guided Greedy Refinement. In all comparisons,
CMC or a high-budget ICE-vMFNM run is used as the accuracy reference, while
true-model ICE-vMFNM serves as the computational baseline.

Unless otherwise specified, the proposed method uses the following default
configuration. The initial surrogate is trained using $M_0=512$ high-fidelity
samples drawn from the nominal density $p$, with $N_{\rm pre}=40{,}000$
pretraining iterations. At each adaptive refinement iteration, a candidate pool
of size $N_c=10^4$ is generated from the current proposal distribution, and
$m_{\rm add}=70$ samples are selected for high-fidelity evaluation. The greedy
selection parameter is set to $\beta=0.5$, and the surrogate is fine-tuned for
$N_{\rm ft}=500$ iterations after each enrichment step. For the two-dimensional
multimodal example, a smaller initial design and a simpler surrogate are used,
as specified in Section~\ref{subsec:four_mode}.

The surrogate consists of an encoder and a prediction network. In the
high-dimensional examples, the encoder has layer widths $[d,40,10]$, mapping the
$d$ dimensional input to a 10-dimensional latent representation, and the
prediction network has layer widths $[10,20,20,1]$. During adaptive refinement,
the last encoder layer is kept fixed to stabilize the latent-space metric used
in the greedy selection rule. The regularization term in Eq.~\eqref{loss_function}
is chosen as $L^2$ regularization, and its weight is determined by the smoothed
gradient-balancing rule with regularization ratio $0.05$.

The proposal distribution is updated using the ICE-vMFNM procedure described in
Section~\ref{sec:surrogate_assisted_proposal_refinement}. In all experiments, the
ICE target coefficient of variation and stopping tolerance are set to
$\delta_{\rm target}=2$ and $\delta_{\rm stop}=2$, respectively. The number of
vMFNM mixture components $K$ is chosen according to the number of dominant
failure modes in each example, and the same $K$ is used for PGGR-ICE-vMFNM and
the corresponding baselines.

Each experiment is repeated $N_{\rm rep}=50$ times. Let $\widehat P_{\mathcal F}$
denote the failure-probability estimate from one independent run. We report the
coefficient of variation and the relative error,
\begin{equation}
\label{eq:error_metrics}
\delta[\widehat P_{\mathcal F}]
=
\frac{\sqrt{\mathbb V[\widehat P_{\mathcal F}]}}
{\mathbb E[\widehat P_{\mathcal F}]},
\qquad
\varepsilon[\widehat P_{\mathcal F}]
=
\frac{
\left|
P_{\mathcal F}^{\rm ref}
-
\mathbb E[\widehat P_{\mathcal F}]
\right|}
{P_{\mathcal F}^{\rm ref}} .
\end{equation}
In practice, the expectation and variance are approximated by the sample mean
and sample variance over the $N_{\rm rep}$ repeated runs.

We also report $N_g$, the average number of actual evaluations of the true
performance function $g(\mathbf u)$ per run. For the proposed method, the
initial surrogate is reused across repeated trials, so the initial training cost
is amortized. If $K_{\rm ad}^{(r)}$ denotes the number of adaptive refinement
iterations in the $r$th run and $K_{\rm ad}
=
\frac{1}{N_{\rm rep}}
\sum_{r=1}^{N_{\rm rep}} K_{\rm ad}^{(r)}$, 
then the reported high-fidelity cost is
\begin{equation}
\label{eq:Ng_cost}
N_g
=
\frac{M_0}{N_{\rm rep}}
+
m_{\rm add}K_{\rm ad}.
\end{equation}
For a single independent run, the unamortized cost is
$N_g^{\rm single}=M_0+m_{\rm add}K_{\rm ad}$.

\subsection{Four mode problem}
\label{subsec:four_mode}

We first consider a two-dimensional benchmark with four failure modes. The
performance function is defined as
\begin{equation*}
g(\mathbf{u})
=
\min
\left\{
\begin{aligned}
&0.1\,(u_1-u_2)^2 - \frac{u_1+u_2}{\sqrt{2}} + 5,\\
&0.1\,(u_1-u_2)^2 + \frac{u_1+u_2}{\sqrt{2}} + 5,\\
&(u_1-u_2) + \frac{7}{\sqrt{2}} + 2,\\
&(u_2-u_1) + \frac{7}{\sqrt{2}} + 2.
\end{aligned}
\right.
\end{equation*}

This performance function induces four separated failure modes, and we therefore
use $K=4$ vMFNM mixture components. The two-dimensional setting allows us to
visualize the adaptive proposal, selected high-fidelity samples, and surrogate
failure boundary, providing a clear illustration of the proposed greedy
enrichment strategy.


For this low-dimensional example, we use a slightly different configuration from
the default setting. The surrogate is a multilayer perceptron with layer widths
\([2,20,20,1]\), and the initial training set contains \(M_0=32\) high-fidelity
samples. At each adaptive refinement iteration, \(m_{\mathrm{add}}=30\) samples
are selected from a candidate pool of size \(3000\). We set \(\beta=1\) to
encourage stronger diversity across the spatially separated failure modes. During
the vMFNM proposal-fitting step, we run 10 parallel EM chains with different
initializations to reduce the risk of mode loss; this does not introduce
additional high-fidelity evaluations. After each enrichment step, the surrogate
is retrained for \(2000\) optimization epochs.

To illustrate the effect of the enrichment strategy, Fig.~\ref{fig:proposal_evolution}
compares greedy and random refinement under comparable evaluation budgets. The
first row corresponds to greedy enrichment and the second row to random
enrichment. Each panel shows the proposal samples, newly selected high-fidelity
samples, accumulated training samples, true failure boundary
\(g(\mathbf u)=0\), and surrogate boundary \(\widehat g_t(\mathbf u)=0\).

\begin{figure}[htbp]
    \centering
    \vspace{0.3cm}

    \begin{overpic}[width=0.30\textwidth]{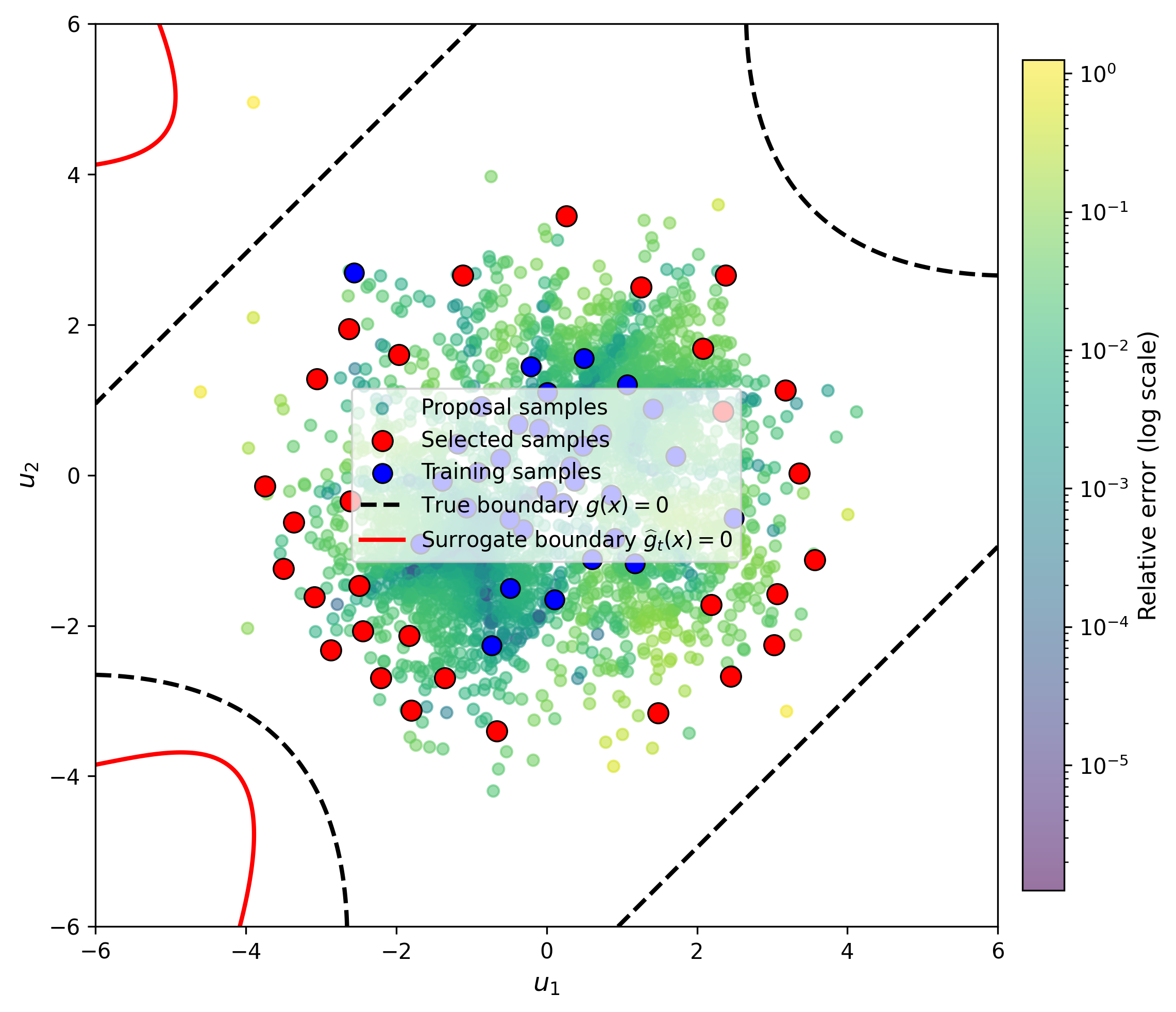}
        \put(25,87){\small\textbf{Iteration 1}}
        \put(-5,43){\makebox(0,0){\rotatebox[origin=c]{90}{\small\textbf{Greedy}}}}
    \end{overpic}
    \begin{overpic}[width=0.30\textwidth]{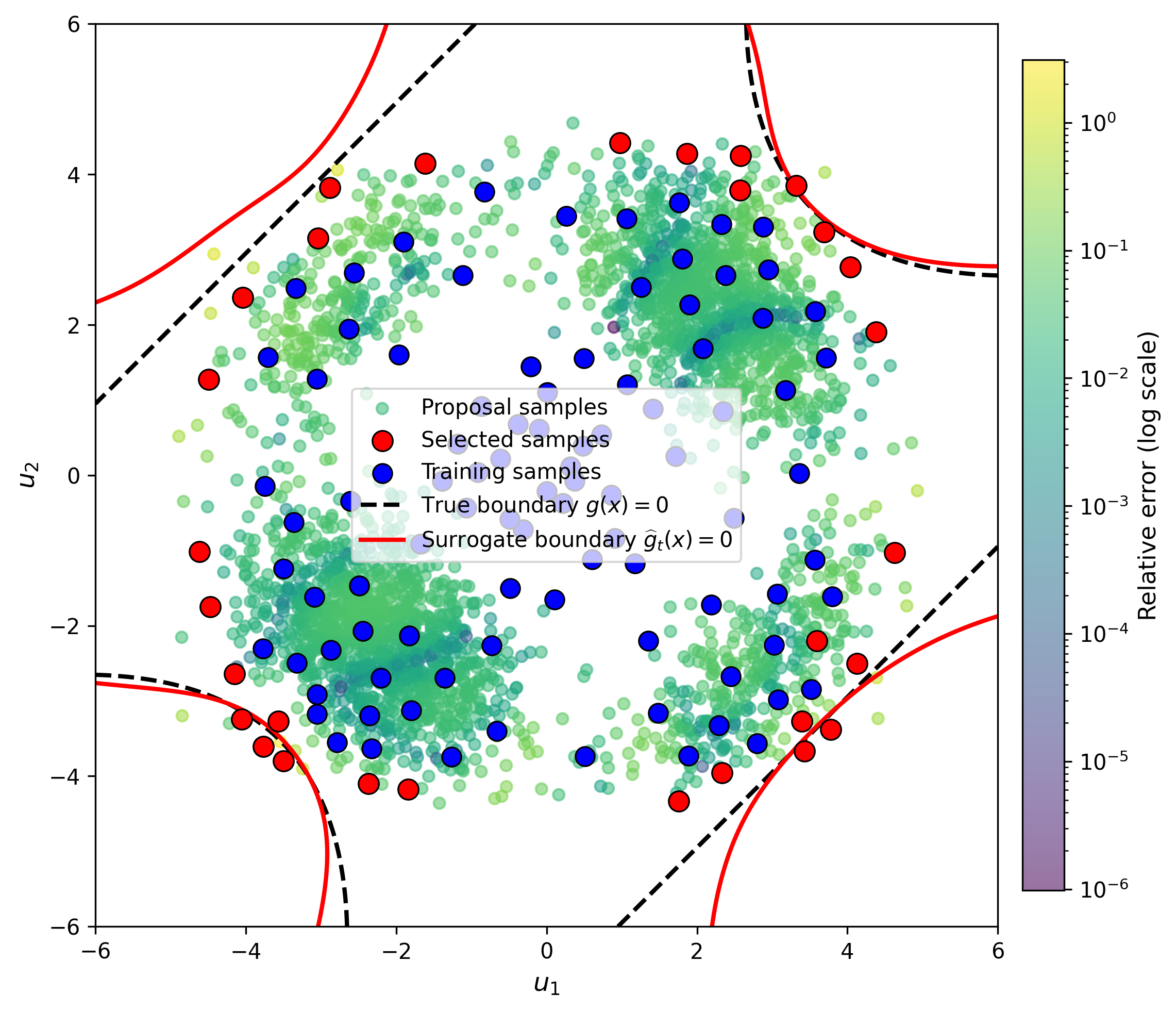}
        \put(25,87){\small\textbf{Iteration 3}}
    \end{overpic}
    \begin{overpic}[width=0.30\textwidth]{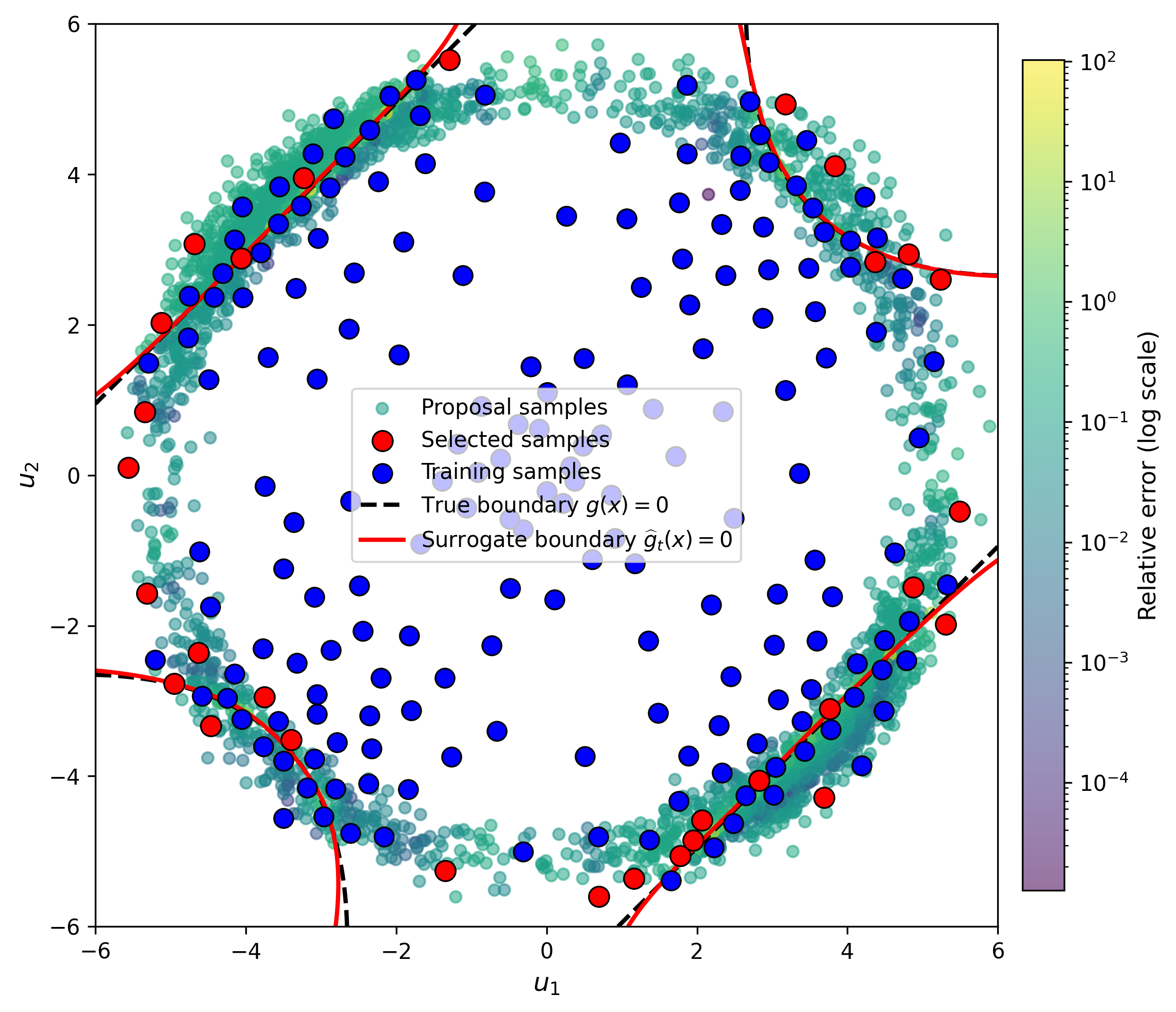}
        \put(25,87){\small\textbf{Iteration 6}}
    \end{overpic}

    \vspace{0.35cm}

    \begin{overpic}[width=0.30\textwidth]{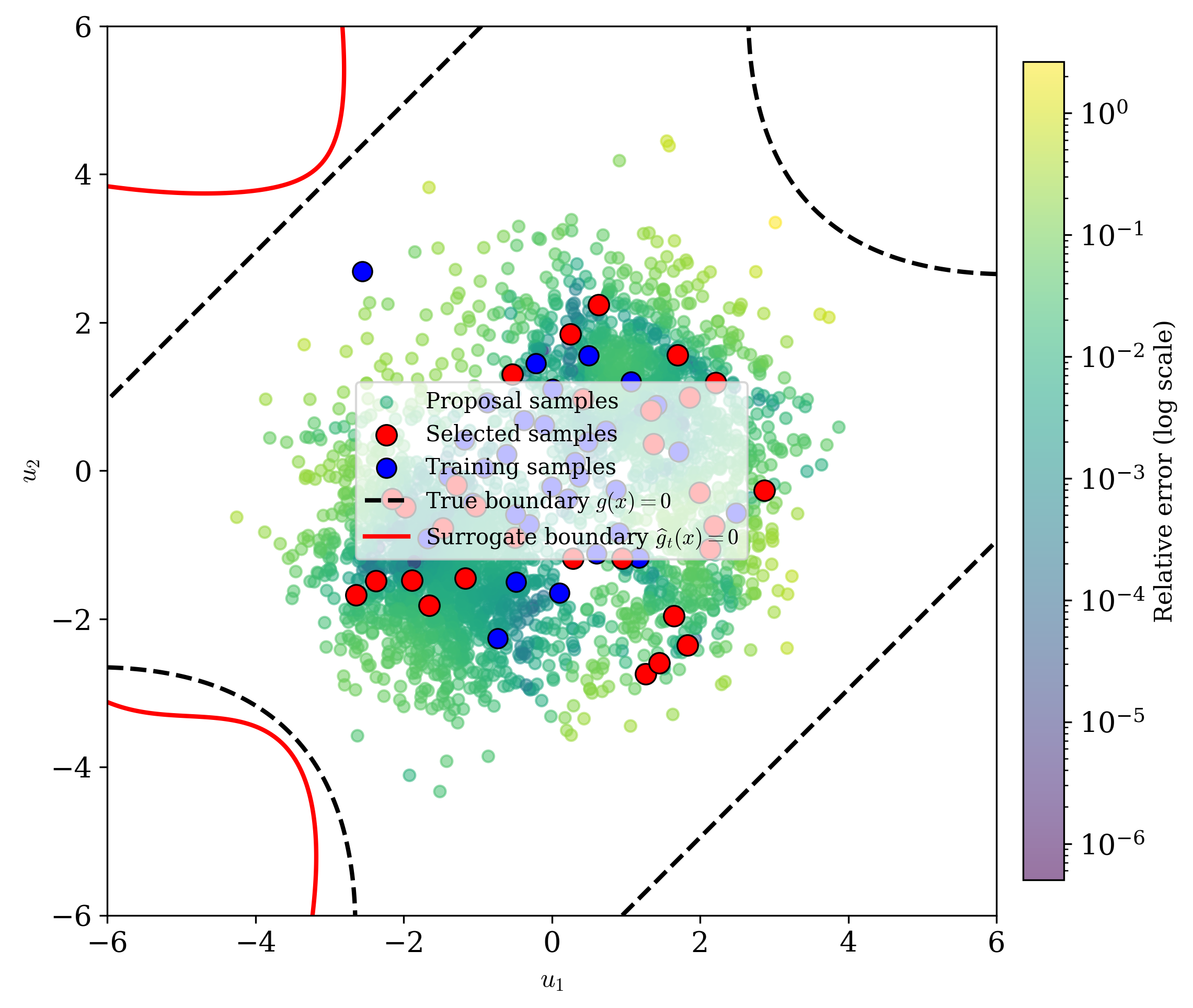}
        \put(-5,43){\makebox(0,0){\rotatebox[origin=c]{90}{\small\textbf{Random}}}}
    \end{overpic}
    \begin{overpic}[width=0.30\textwidth]{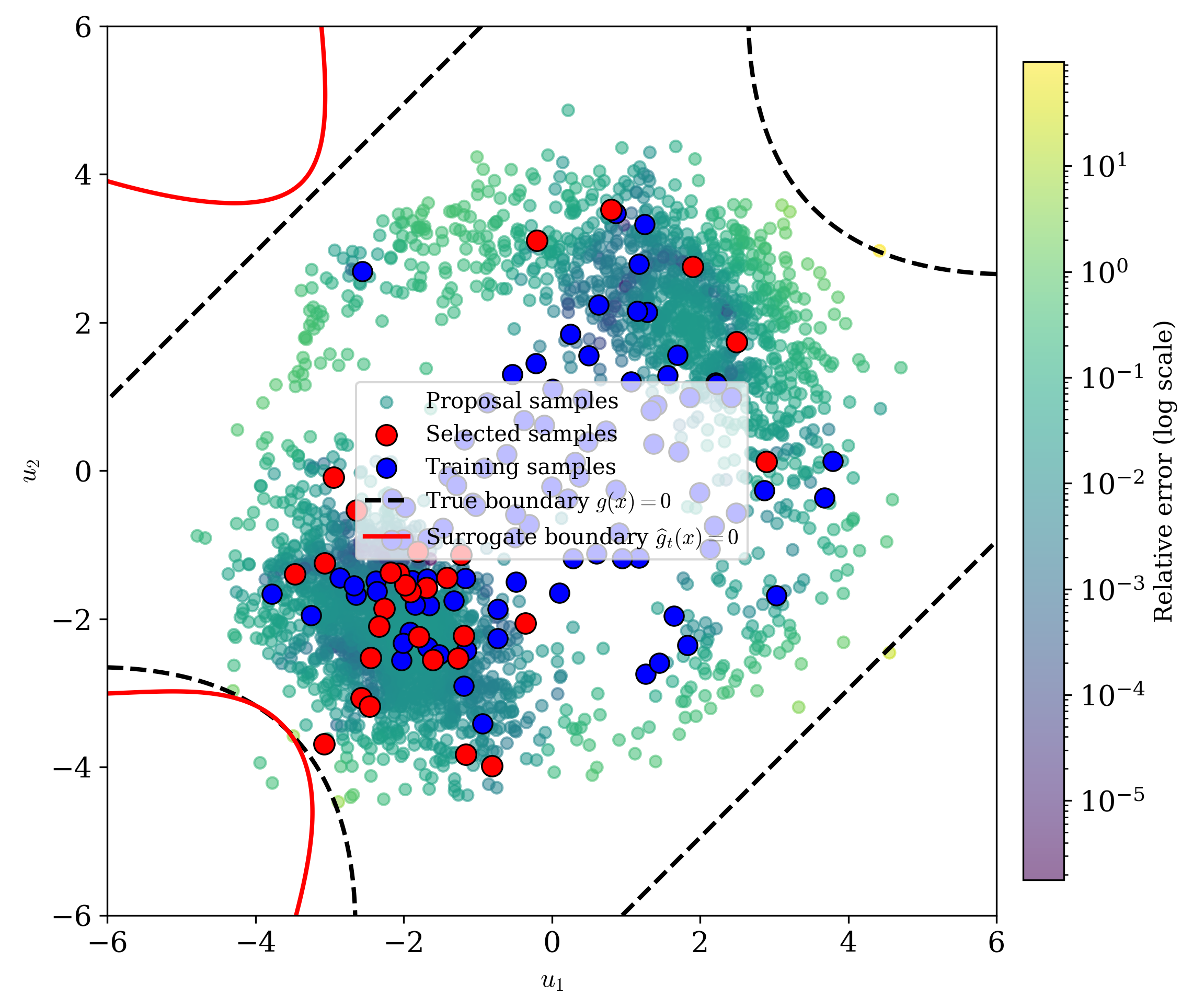}
    \end{overpic}
    \begin{overpic}[width=0.30\textwidth]{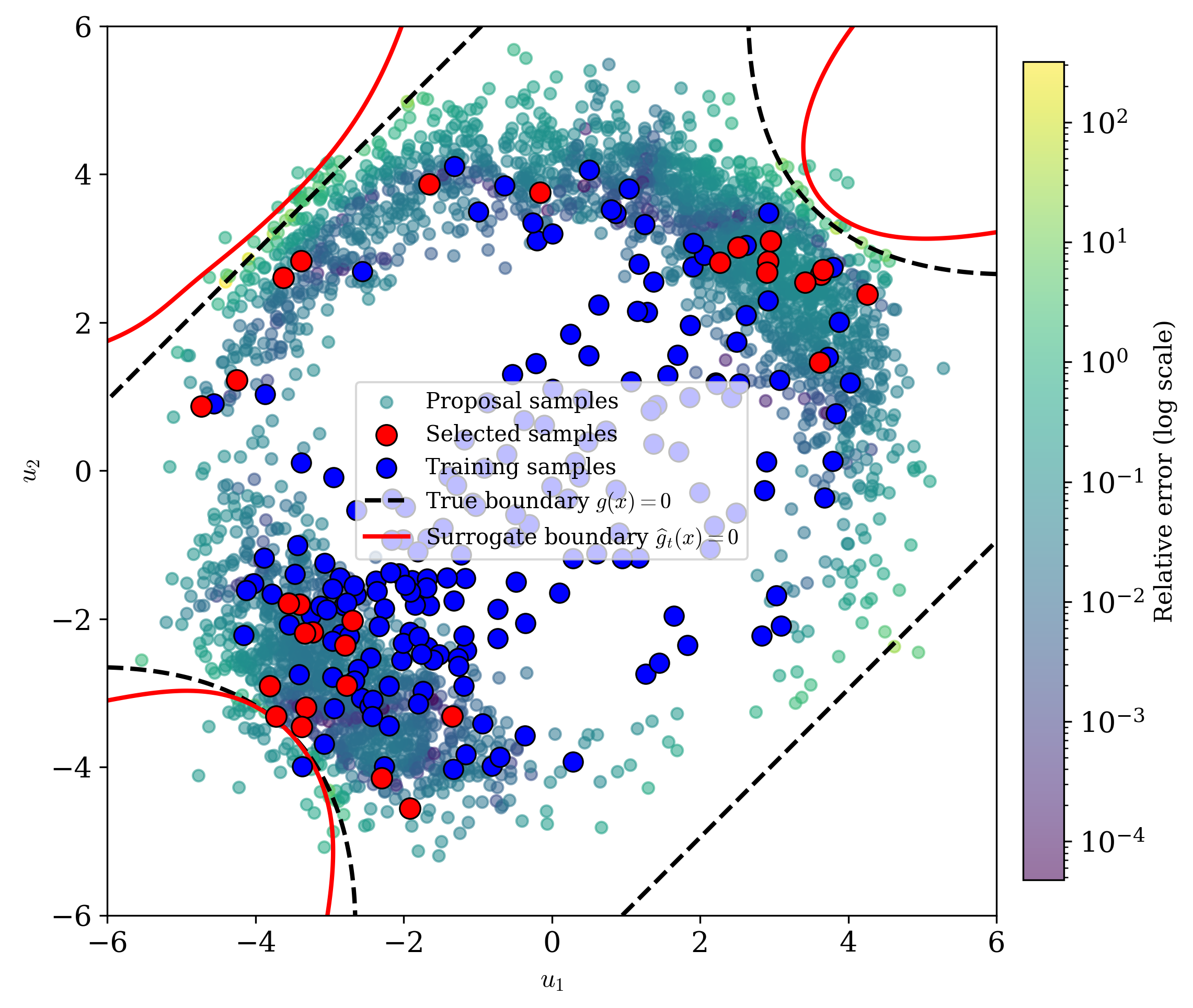}
    \end{overpic}

    \caption{Evolution of the adaptive proposal samples, selected high-fidelity
    samples, accumulated training samples, and surrogate failure boundaries for
    the two-dimensional multimodal example. The first row uses the proposed
    greedy enrichment strategy, while the second row uses random enrichment.}
    \label{fig:proposal_evolution}
\end{figure}

The difference between the two enrichment strategies is evident from
Fig.~\ref{fig:proposal_evolution}. Greedy enrichment selects samples near the
evolving failure boundary while maintaining coverage of different failure modes,
leading to more targeted surrogate correction. In contrast, random enrichment
often places samples away from informative boundary regions, so some
failure-boundary components are not refined effectively. This supports the use
of a boundary-aware and diversity-promoting enrichment strategy.

\begin{table}[htbp]
\centering
\caption{Comparison of different methods for the two-dimensional multimodal
example. Here \(N_g\) denotes the average number of actual evaluations of the
performance function \(g(\mathbf u)\) per run. The crude Monte Carlo result is
used as the reference value \(P_{\mathcal F}^{\mathrm{ref}}\).}
\label{tab:comparison_2d_multimodal}
\begin{tabular}{lcccc}
\toprule
Method 
& \(N_g\)
& \(\overline P_{\mathcal F}\)
& \(\varepsilon_{\mathrm{rel}}\)
& \(\delta\) \\
\midrule
CMC 
& \(1.0\times 10^{9}\)
& \(1.21\times 10^{-6}\)
& -- 
& -- \\

ICE-vMFNM
& \(1.1\times 10^4\)
& \(1.20\times 10^{-6}\)
& \(0.008\)
& \(0.048\) \\

PGGR-ICE-vMFNM
& \(1.9\times 10^{2}\)
& \(1.24\times 10^{-6}\)
& \(0.024\)
& \(0.052\) \\

Random-ICE-vMFNM
& \(2.1\times 10^{2}\)
& \(7.60\times 10^{-7}\)
& \(0.387\)
& \(0.244\) \\
\bottomrule
\end{tabular}
\end{table}

Table~\ref{tab:comparison_2d_multimodal} further quantifies the effect of the
proposed greedy enrichment strategy. The true-model ICE-vMFNM method gives an
estimate close to the crude Monte Carlo reference, with relative error \(0.008\)
and coefficient of variation \(0.048\), but requires \(1.1\times 10^4\)
evaluations of the performance function. In comparison, PGGR-ICE-vMFNM achieves
a comparable estimate, with relative error \(0.024\) and coefficient of
variation \(0.052\), while using only \(1.9\times 10^2\) high-fidelity
evaluations on average.

The random enrichment baseline uses a similar number of high-fidelity
evaluations but is much less accurate, with relative error \(0.387\) and
coefficient of variation \(0.244\). This is consistent with
Fig.~\ref{fig:proposal_evolution}, where random samples are less concentrated
near informative failure-boundary regions. These results show that the accuracy
of PGGR-ICE-vMFNM comes not merely from adding high-fidelity samples, but from
selecting informative and diverse samples along the evolving proposal
distribution.

\subsection{High-dimensional stochastic diffusion problem}
\label{subsec:stochastic_diffusion_example}

We next consider a high-dimensional rare-event problem governed by a
one-dimensional diffusion equation with a stochastic diffusion coefficient
\cite{doi:10.1137/21M1404119}. Let \(D=(0,1)\). For almost every \(\omega\in\Omega\), we seek the weak solution
\(y(\cdot,\omega)\in V\), where $V=\{v\in H^1(D): v(0)=0\}$,
such that
\begin{equation}
    \int_D
    a(x,\omega)
    \frac{\partial y(x,\omega)}{\partial x}
    \frac{\partial v(x)}{\partial x}
    \,dx
    =
    \int_D v(x)\,dx,
    \qquad
    \forall v\in V.
    \label{eq:diffusion_weak_form}
\end{equation}
This weak formulation corresponds to the boundary conditions $y(0,\omega)=0,
    \frac{\partial y}{\partial x}(1,\omega)=0$,
where the homogeneous Neumann condition at \(x=1\) is imposed naturally.
The diffusion coefficient is modeled as a log-normal random field,
\(a(x,\omega)=\exp(Z(x,\omega))\). The random field is specified such that
\(\mathbb{E}[a(x,\cdot)]=1\) and \(\mathrm{Std}[a(x,\cdot)]=0.1\).
Accordingly, the mean and variance of the underlying Gaussian field \(Z\) are
given by $\mu_Z
    =
    \log\bigl(\mathbb{E}[a(x,\cdot)]\bigr)
    -
    \sigma_Z^2/2, 
    \;
    \sigma_Z^2
    =
    \log
    \left(
    1+
    \mathrm{Std}[a(x,\cdot)]^2/
    \mathbb{E}[a(x,\cdot)]^2
    \right)$. 
The covariance function of \(Z\) is chosen as $c(x,y)
    =
    \sigma_Z^2
    \exp
    \left(
    -
    |x-y|/\lambda
    \right),
    \lambda = 0.01$. 
To obtain a finite-dimensional parameterization, we use the truncated
Karhunen--Lo\`eve expansion
    $Z_d(x,\omega) = \mu_Z
    +
    \sigma_Z
    \sum_{m=1}^{d}
    \sqrt{\nu_m}\,
    \theta_m(x)
    U_m(\omega)$, 
where \(\{(\nu_m,\theta_m)\}_{m=1}^d\) are the eigenpairs associated with the
covariance operator, and $U_1,\ldots,U_d
    \sim
    \mathcal{N}(0,1)$ 
are independent standard normal random variables. The truncated coefficient is
then defined as $a_d(x,\omega)
    =
    \exp(Z_d(x,\omega))$. 
Following the reference setting, we choose \(d=100\), which captures
approximately \(81\%\) of the variability of \(\log a\).

The weak problem is discretized by continuous piecewise linear finite elements
on a uniform mesh with mesh size \(h=1/512\). Let \(y_h\) denote the corresponding
finite element solution. The failure event is defined by the solution value at
the right endpoint. Specifically, after the finite-dimensional KL
parameterization, failure occurs when $y_h(1;\mathbf u)>0.535$. Representative solution sample paths are plotted in Fig.~\ref{fig:poisson_solution}.
Thus, the performance function is defined as
\begin{equation}
    g(\mathbf u)
    =
    0.535-y_h(1;\mathbf u),
    \qquad
    \mathbf u=(U_1,\ldots,U_d)\in\mathbb R^d.
    \label{eq:diffusion_lsf}
\end{equation}

\begin{figure}[htbp]
    \centering
    \includegraphics[width=0.45\linewidth]{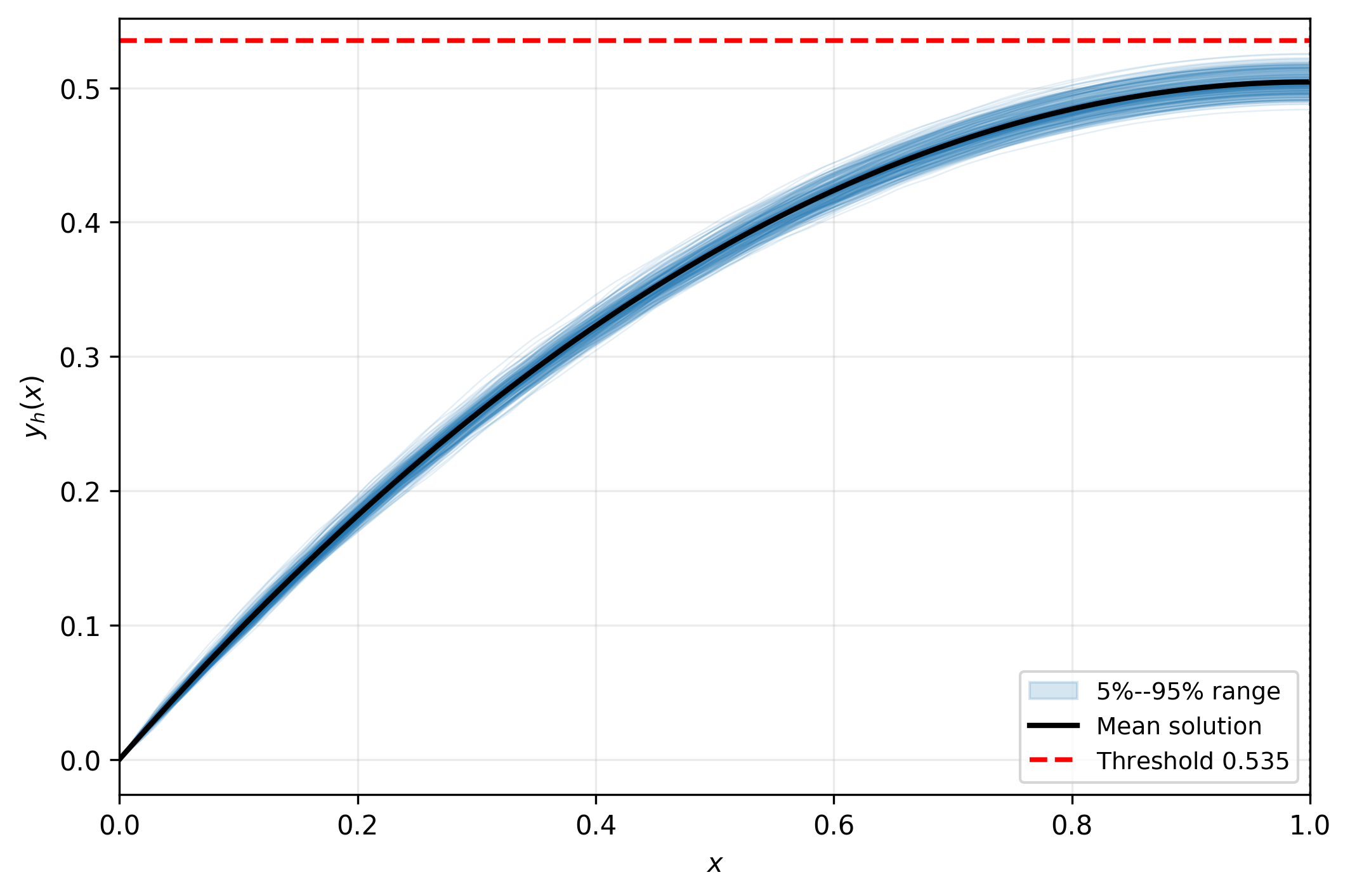}
    \vspace{-0.2cm}
    \caption{Sample paths of the finite element solution for the stochastic
    diffusion problem. The shaded region indicates the \(5\%\)--\(95\%\) range,
    the black curve denotes the mean solution, and the red dashed line marks the
    failure threshold \(0.535\).}
    \label{fig:poisson_solution}
\end{figure}

This presents a 100-dimensional PDE-driven rare-event problem in which each
evaluation of \(g(\mathbf u)\) requires a finite element solve. 

\begin{figure}[htbp]
    \centering

    \begin{overpic}[width=0.295\textwidth]{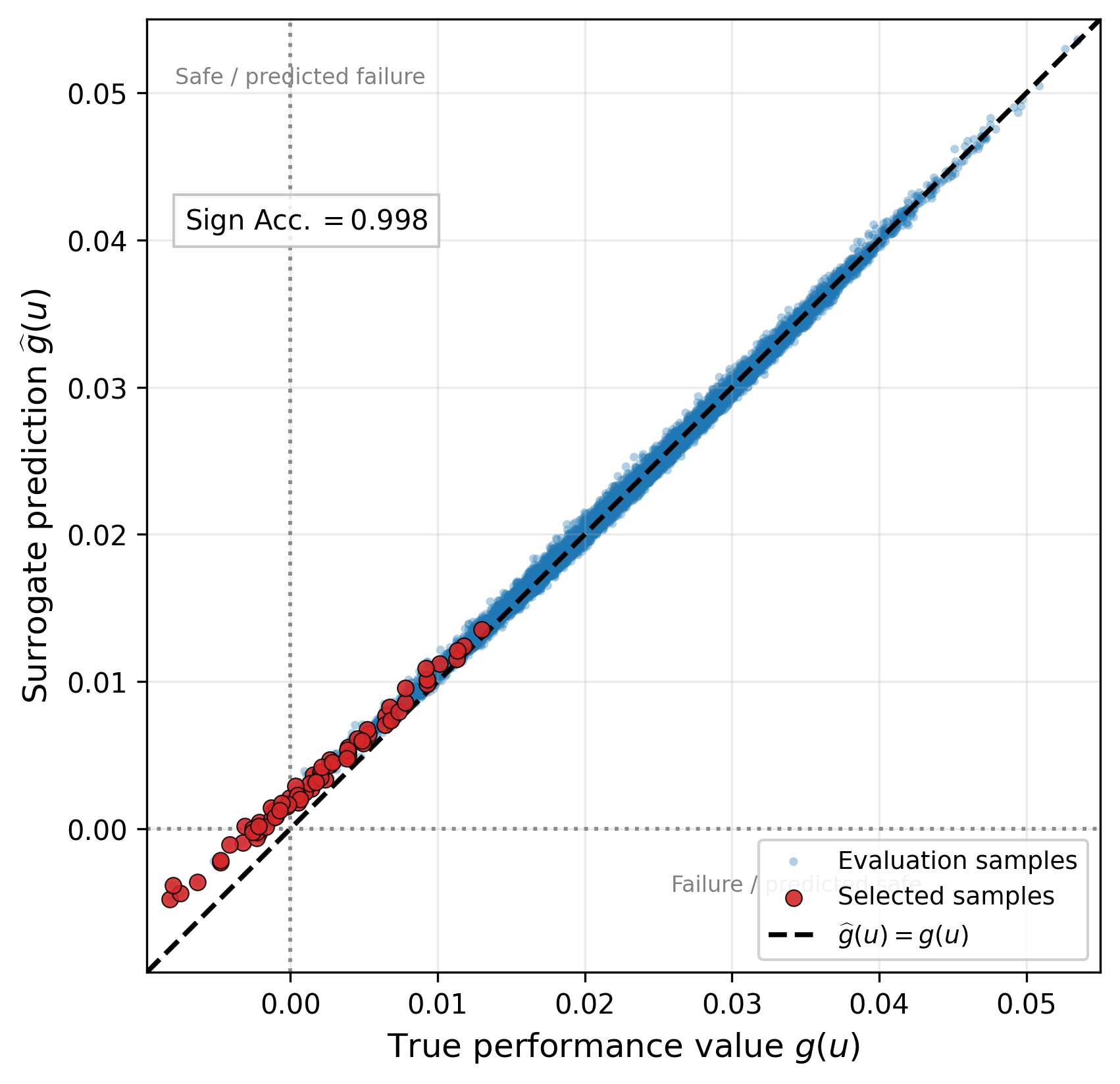}
        \put(37,98){\small Iteration 1}
    \end{overpic}
    \begin{overpic}[width=0.30\textwidth]{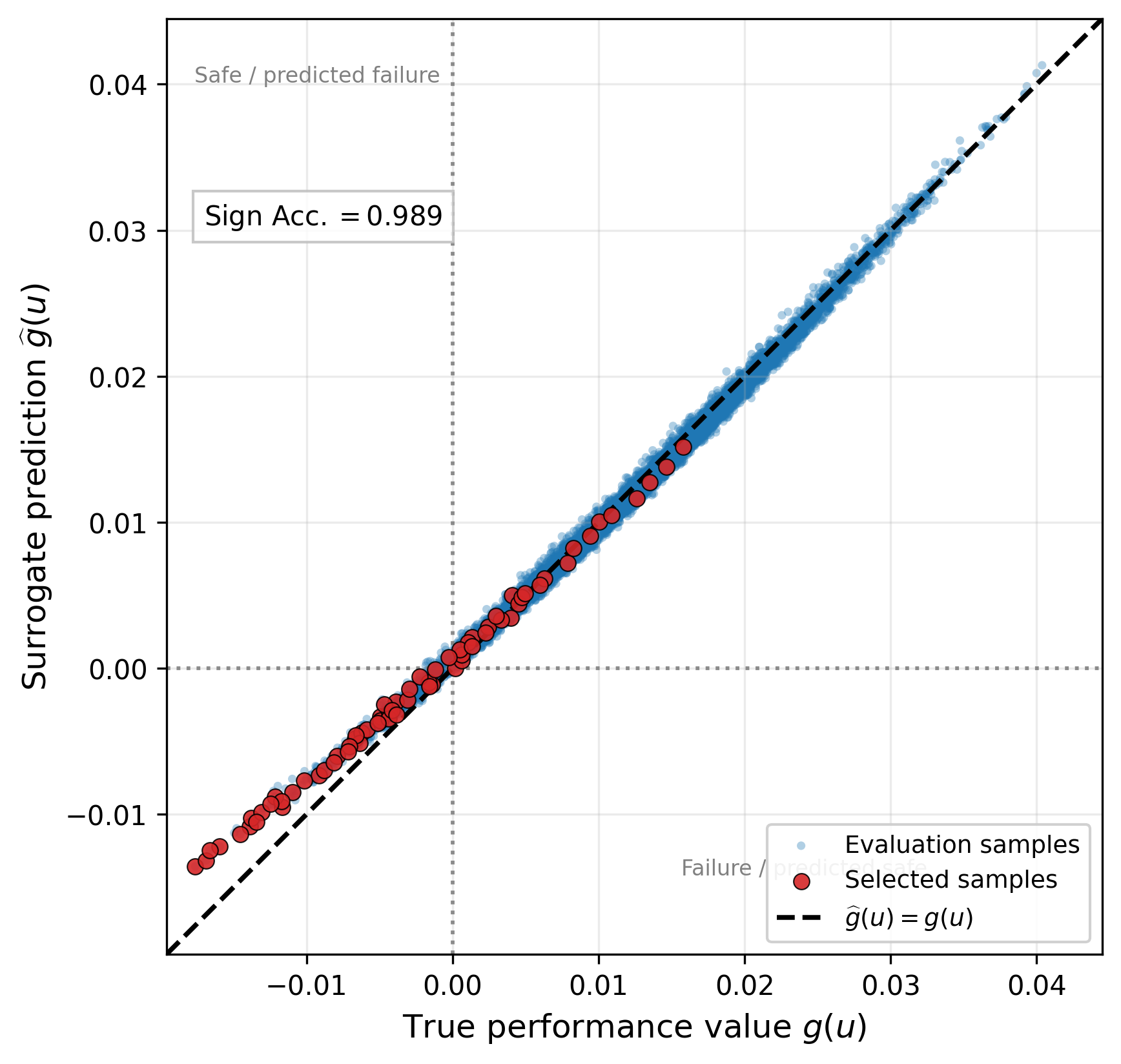}
        \put(37,96){\small Iteration 2}
    \end{overpic}
    \begin{overpic}[width=0.30\textwidth]{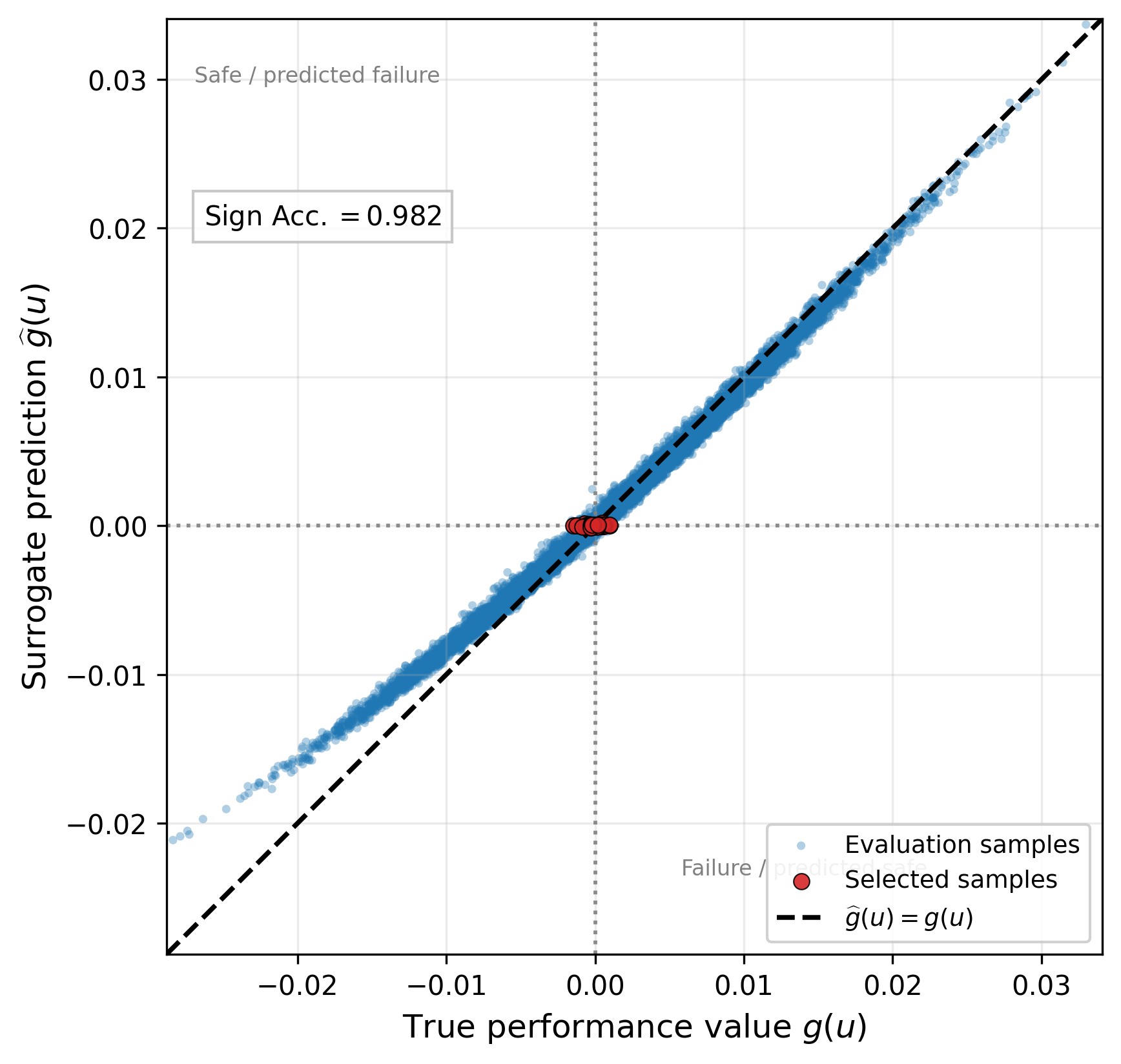}
        \put(37,96){\small Iteration 3}
    \end{overpic}

    \vspace{0.2em}

    \begin{overpic}[width=0.30\textwidth]{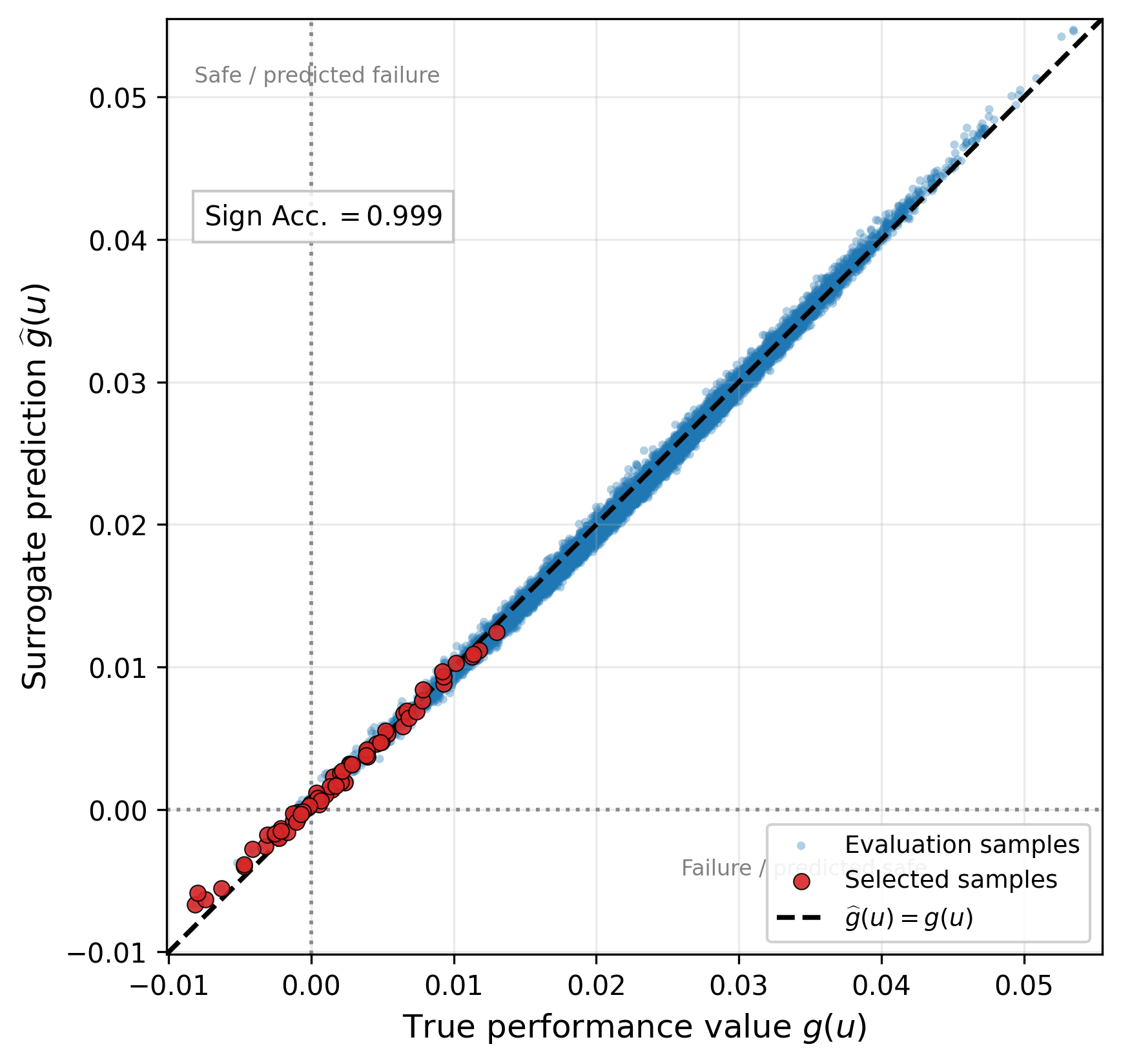}
    \end{overpic}
    \begin{overpic}[width=0.30\textwidth]{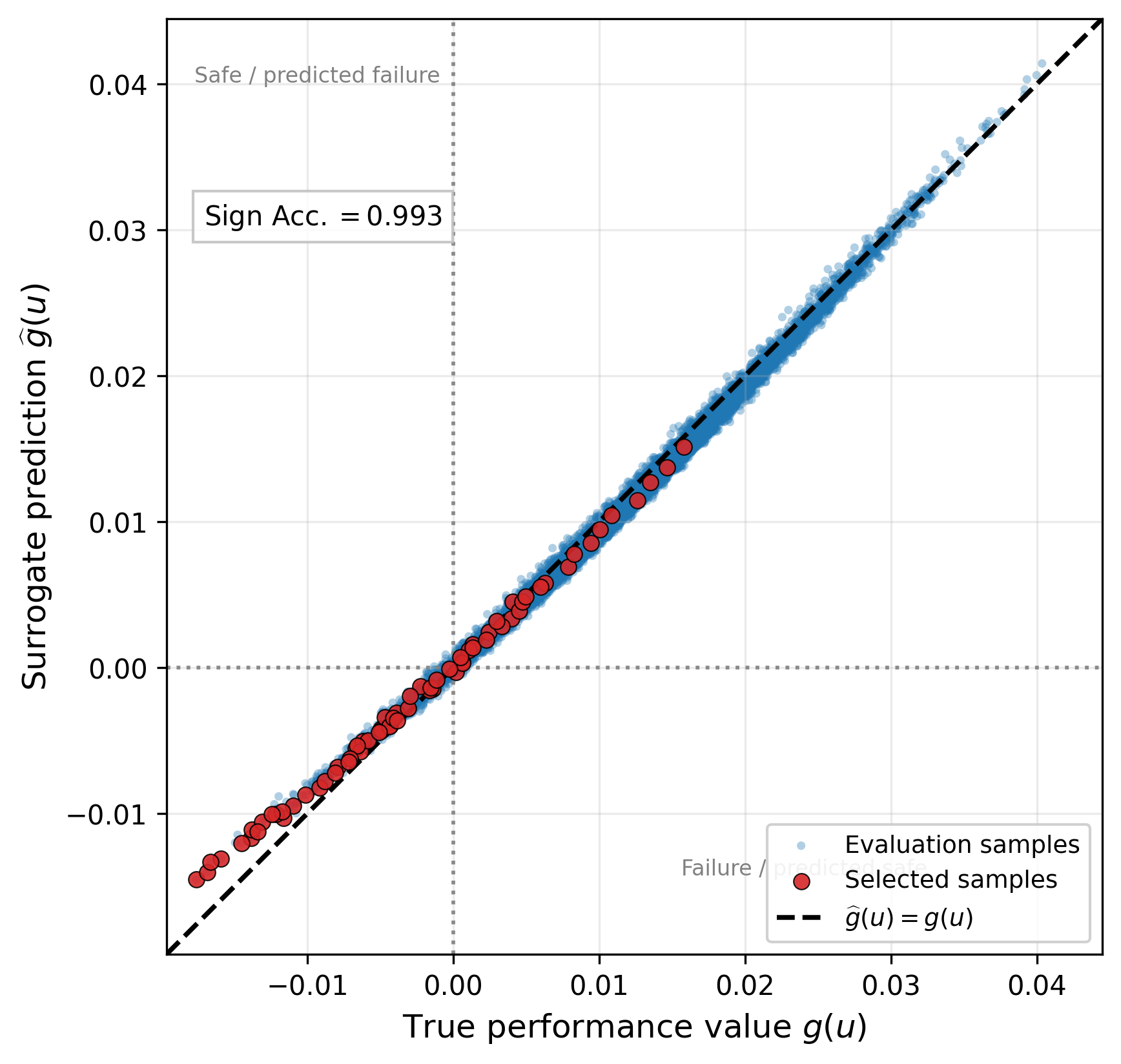}
    \end{overpic}
    \begin{overpic}[width=0.30\textwidth]{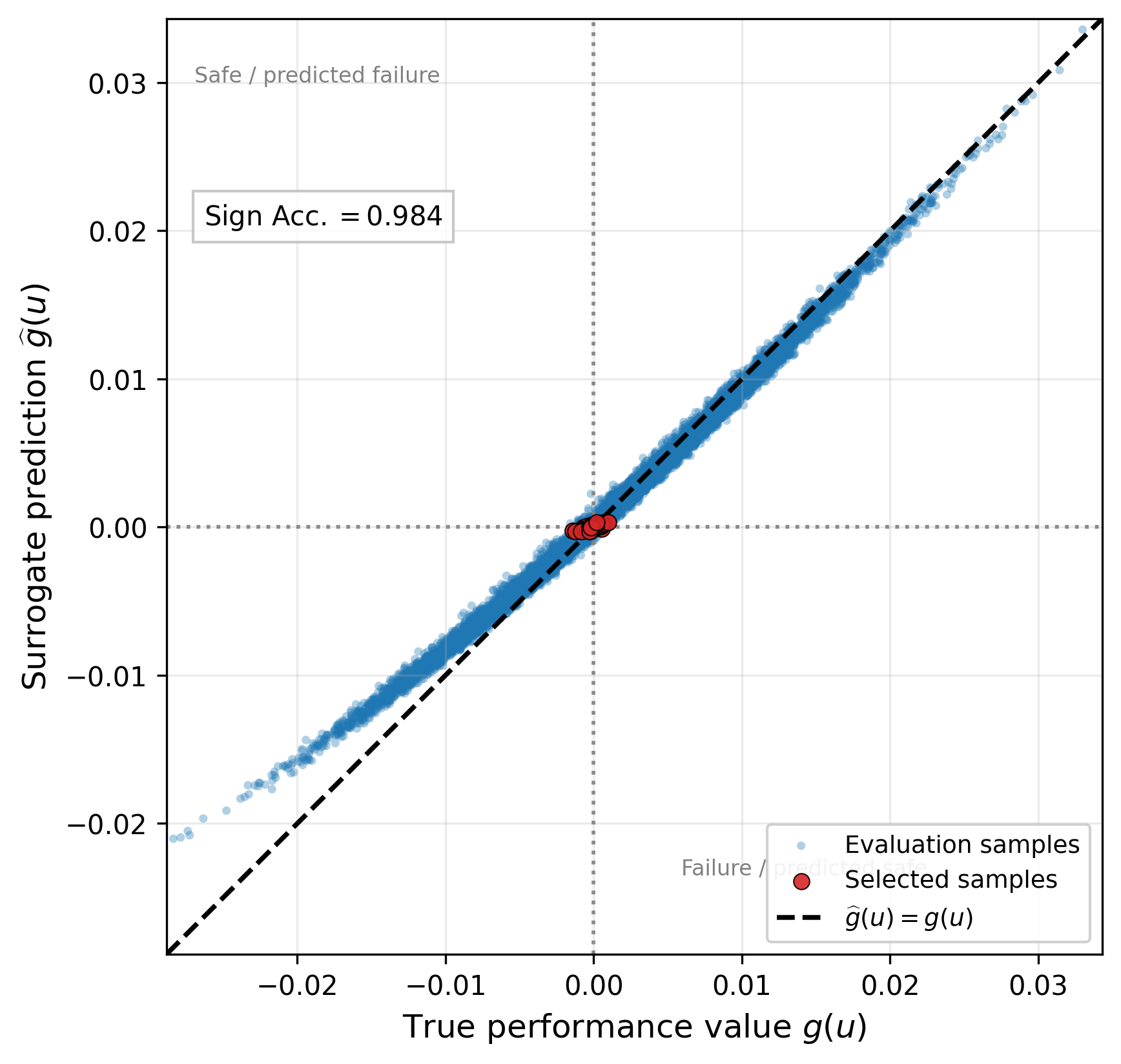}
    \end{overpic}

    \caption{Surrogate prediction quality before and after local refinement for
    the stochastic diffusion problem. The top and bottom rows show predictions
    before and after adding the selected high-fidelity samples, respectively.
    Each column corresponds to one adaptive iteration. The dashed line denotes
    \(\widehat g(\mathbf u)=g(\mathbf u)\), and the dotted lines indicate the
    decision boundaries \(g(\mathbf u)=0\) and \(\widehat g(\mathbf u)=0\).}
    \label{fig:diffusion_surrogate_refinement}
\end{figure}

Figure~\ref{fig:diffusion_surrogate_refinement} illustrates the effect of local
surrogate refinement. Since the failure boundary cannot be visualized directly
in 100 dimensions, we compare \(g(\mathbf u)\) and \(\widehat g(\mathbf u)\) on
samples drawn from the current proposal. The selected samples concentrate near
the decision boundary \(g(\mathbf u)=0\), where sign errors directly affect the
failure indicator. After refinement, the surrogate predictions become more
aligned with the true values in this boundary region, showing that the proposed
method improves the surrogate locally in the proposal-induced region relevant to
the ICE update.

\begin{table}[htbp]
\centering
\caption{Performance comparison for the 100-dimensional diffusion example.
The crude Monte Carlo estimate is used as the reference value
\(P_{\mathcal F}^{\mathrm{ref}}\). Here \(N_g\) denotes the average number of
actual evaluations of the performance function \(g(\mathbf u)\) per run.}
\label{tab:comparison_100_dimensional}
\begin{tabular}{lcccc}
\toprule
Method
& \(N_g\)
& \(\overline P_{\mathcal F}\)
& \(\varepsilon_{\mathrm{rel}}\)
& \(\delta\) \\
\midrule
CMC
& \(1.0\times 10^{7}\)
& \(1.39\times 10^{-4}\)
& --
& -- \\

ICE-vMFNM
& \(4.1\times 10^3\)
& \(1.43\times 10^{-4}\)
& \(0.029\)
& \(0.082\) \\

AK-MCS
& \(2.2\times 10^2\)
& \(0\)
& \(1.000\)
& -- \\

PGGR-ICE-vMFNM
& \(2.2\times 10^{2}\)
& \(1.35\times 10^{-4}\)
& \(0.029\)
& \(0.035\) \\

Random-ICE-vMFNM
& \(2.2\times 10^{2}\)
& \(1.09\times 10^{-4}\)
& \(0.222\)
& \(0.066\) \\
\bottomrule
\end{tabular}
\end{table}

Table~\ref{tab:comparison_100_dimensional} reports the results for the
100-dimensional diffusion example. We also include AK-MCS~\cite{echard2011ak} as
a classical Gaussian-process/Kriging-based reliability baseline with learning
function \(U(\mathbf u)=|\mu(\mathbf u)|/\sigma(\mathbf u)\). In our
implementation, the standard stopping criterion \(U_{\min}\ge 2\) is already
satisfied after the initial design with \(N_g=220\) evaluations. However, the
resulting Kriging surrogate classifies no samples in the Monte Carlo population
as failures, giving \(\overline P_{\mathcal F}=0\).
This indicates that a global Kriging surrogate built from the initial design
fails to identify the failure boundary in this high-dimensional rare-event
problem.

In contrast, PGGR-ICE-vMFNM gives
\(\overline P_{\mathcal F}=1.35\times 10^{-4}\), close to the CMC reference value
\(1.39\times 10^{-4}\), with relative error \(0.029\). It reduces the average
number of high-fidelity evaluations from \(4.1\times 10^3\) for true-model
ICE-vMFNM to \(2.2\times 10^2\), while Random-ICE-vMFNM uses a comparable budget
but substantially underestimates the failure probability. These results show
that the improvement comes from proposal-guided greedy refinement rather than
merely from adding more training data, and highlight the benefit of refining the
surrogate along the evolving proposal distribution.

\subsection{Semilinear heat equation with random heat source}
\label{subsec:semilinear_heat_example}

We next consider a one-dimensional semilinear heat equation with a random heat
source. This example is used to test the proposed method on a nonlinear
time-dependent PDE-driven rare-event problem.

For a given source term \(f(x)\), the temperature field \(y(t,x)\) satisfies
\begin{equation}
    \partial_t y(t,x)
    -
    \nu \partial_{xx} y(t,x)
    +
    \gamma y(t,x)^3
    =
    f(x),
    \qquad
    (t,x)\in (0,T]\times(0,1),
    \label{eq:semilinear_heat_pde}
\end{equation}
with zero Dirichlet boundary conditions and zero initial condition. 
In the numerical experiment, we set $T=1, \nu=0.02, \gamma=1$. We introduce uncertainty through the source term \(f\). Specifically, the source
term is modeled as a lognormal random field. Let $\mathbf u=(u_1,\ldots,u_d)^\top\sim \mathcal N(\mathbf 0,I_d)$. For each realization of \(\mathbf u\), the source field is defined by $f(x;\mathbf u)
    =
    f_0(x)
    \exp\left(
        \sigma_f Z_d(x;\mathbf u)
        -
        \frac{1}{2}\sigma_f^2
        \operatorname{Var}\bigl[Z_d(x;\mathbf u)\bigr]
    \right)$,
where the deterministic mean profile is $
    f_0(x)
    =
    5\exp\left(-80(x-0.5)^2\right)$.
The underlying Gaussian random field \(Z(x)\) is assumed to have zero mean and
exponential covariance kernel
\begin{equation}
    C(x,x')
    =
    \exp\left(
        -\frac{|x-x'|}{\ell}
    \right),
    \qquad \ell=0.1 .
    \label{eq:semilinear_heat_covariance}
\end{equation}
To obtain a finite-dimensional parameterization, we use the truncated
Karhunen--Lo\`eve expansion $Z_d(x;\mathbf u)
    =
    \sum_{j=1}^{d}
    \sqrt{\lambda_j}\phi_j(x)u_j$, 
where \(\{(\lambda_j,\phi_j)\}_{j=1}^{d}\) are the leading eigenpairs associated
with the covariance operator. In this experiment, we retain \(d=100\) KL modes,
so that the random source is parameterized by a 100-dimensional standard
Gaussian vector. We set \(\sigma_f=0.6\). For each realization of
\(\mathbf u\), the source \(f(x;\mathbf u)\) is substituted into
\eqref{eq:semilinear_heat_pde}, and the corresponding solution is denoted by
\(y(t,x;\mathbf u)\). The quantity of interest is the accumulated thermal exposure over the
observation region $D_{\mathrm{obs}}=(0.4,0.6)$. It is defined by
\begin{equation}
    Q(\mathbf u)
    =
    \frac{1}{T|D_{\mathrm{obs}}|}
    \int_0^T
    \int_{D_{\mathrm{obs}}}
    y(t,x;\mathbf u)\,dx\,dt .
    \label{eq:semilinear_heat_qoi}
\end{equation}
The rare event is defined as the event that the accumulated thermal exposure
exceeds the threshold \(z_{\mathrm{heat}}=2.4\). Equivalently, we define
the performance function $g(\mathbf u)
    =
    2.4 - Q(\mathbf u)$, 
so that the failure domain is $
    \Omega_{\mathcal F}
    =
    \{\mathbf u:g(\mathbf u)\le 0\}
    =
    \{\mathbf u:Q(\mathbf u)\ge 2.4\}$.

The PDE is discretized by a finite difference method with \(n_x=256\) interior
spatial grid points and \(n_t=200\) uniform time steps. We use a backward Euler
scheme in time and a centered finite difference discretization in space. At each
time step, the nonlinear algebraic system is solved by Newton's method with
tolerance \(10^{-8}\) and a maximum of 12 Newton iterations. A representative solution is plotted in Fig.~\ref{fig:semilinear_heat}. This example is challenging because the input dimension is \(d=100\),
the governing equation is nonlinear, and each evaluation of \(g(\mathbf u)\)
requires solving a time-dependent PDE.

\begin{figure}[htbp]
    \centering
    \includegraphics[width=0.5\linewidth]{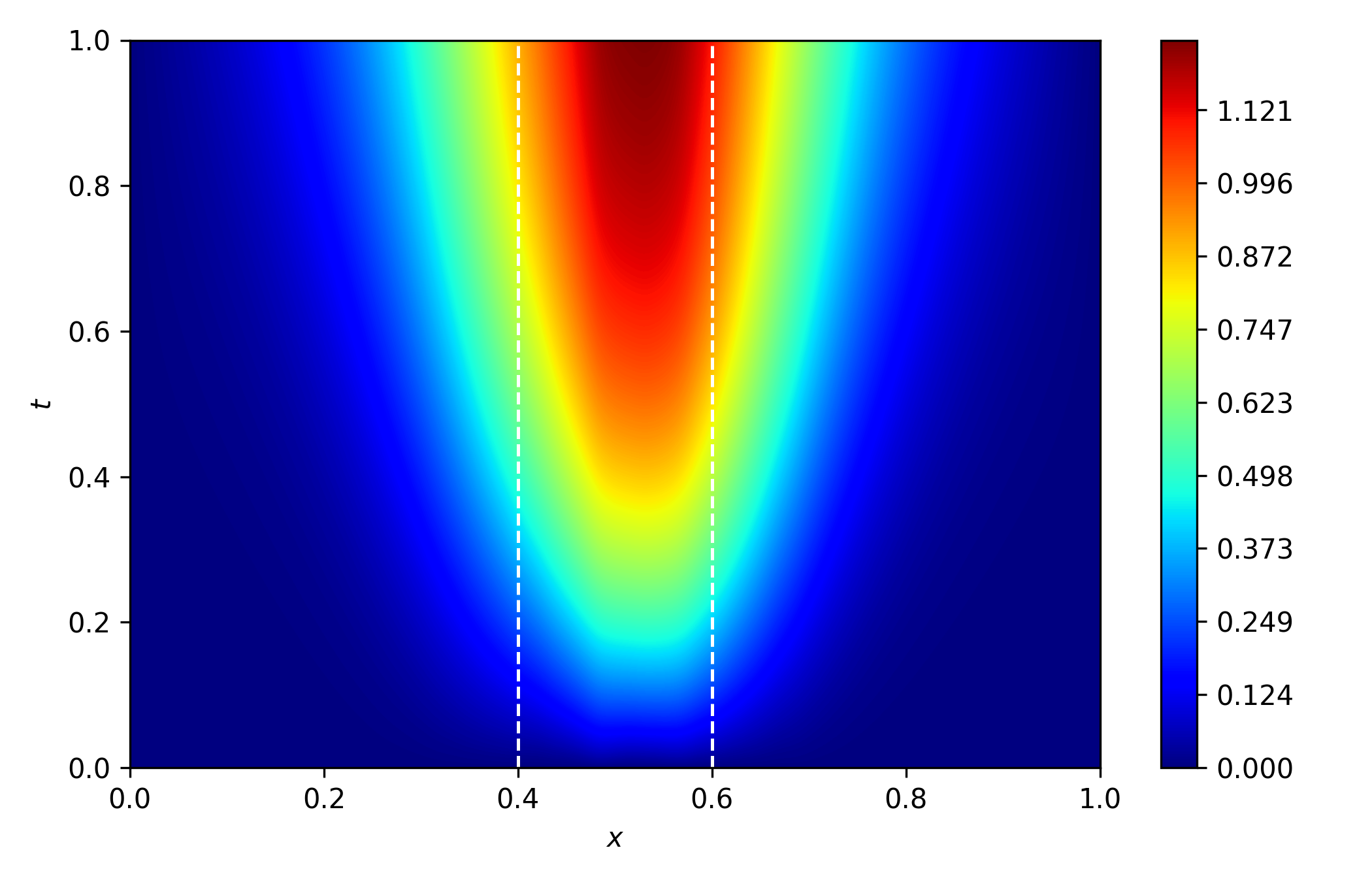}
    \vspace{-0.3cm}
    \caption{A representative solution of the semilinear heat equation with a
random heat source.}
    \label{fig:semilinear_heat}
\end{figure}

\begin{table}[htbp]
\centering
\caption{Performance comparison for the 100-dimensional stochastic semilinear
heat equation example. The reference value is obtained by the high-budget
ICE-vMFNM method.}
\label{tab:comparison_semilinear_heat}
\begin{tabular}{lcccc}
\toprule
Method 
& $N_g$
& $\overline P_{\mathcal F}$
& $\varepsilon_{\mathrm{rel}}$
& $\delta$ \\
\midrule
ICE-vMFNM (reference)  
& $1.3\times 10^{5}$
& $5.93\times 10^{-5}$
& -- 
& 0.014 \\

ICE-vMFNM
& $4.2\times 10^3$
& $6.00\times 10^{-5}$
& $0.012$
& $0.088$ \\

PGGR-ICE-vMFNM
& $2.2\times 10^{2}$
& $5.86\times 10^{-5}$
& $0.012$
& $0.108$ \\

Random-ICE-vMFNM 

& $2.9\times 10^{2}$
& $2.2\times 10^{-5}$
& $0.629$
& $0.267$ \\ 
\bottomrule
\end{tabular}
\end{table}

Table~\ref{tab:comparison_semilinear_heat} reports the results for the 100-dimensional
stochastic semilinear heat equation example. The high-budget ICE-vMFNM estimate
is used as the reference value, giving
$P_{\mathcal F}^{\rm ref}=5.93\times 10^{-5}$. The reduced-budget true-model
ICE-vMFNM method produces a close estimate,
$\widehat P_{\mathcal F}=6.00\times 10^{-5}$, with relative error
$\varepsilon_{\rm rel}=0.012$, but requires $4.2\times 10^3$ evaluations of the
performance function. In comparison, PGGR-ICE-vMFNM achieves essentially the
same relative error, also $\varepsilon_{\rm rel}=0.012$, using only
$2.2\times 10^2$ high-fidelity evaluations on average. This corresponds to an
approximately 19-fold reduction in true-model evaluations. Its coefficient of
variation is slightly larger than that of ICE-vMFNM, but remains reasonable
given the substantial reduction in computational cost.

The Random-ICE-vMFNM baseline uses a comparable number of high-fidelity
evaluations, but severely underestimates the failure probability, yielding
$\widehat P_{\mathcal F}=2.2\times 10^{-5}$ and
$\varepsilon_{\rm rel}=0.629$. Its coefficient of variation is also much larger,
with $\delta=0.267$. This comparison indicates that the improvement of
PGGR-ICE-vMFNM is not simply due to adding a small number of high-fidelity
samples, but to selecting informative samples through the proposal-guided greedy
refinement strategy. For this nonlinear time-dependent PDE example, the proposed
method therefore maintains an accurate surrogate in the failure-relevant region
while substantially reducing the number of expensive model evaluations.

\subsection{Heat conduction problem}

Finally, we consider a heat conduction problem adapted from
\cite{PAPAIOANNOU2019106564}. The computational domain is
$D=(-0.5,0.5)\,\mathrm{m}\times(-0.5,0.5)\,\mathrm{m}$, and the temperature field
$T(\mathbf{x})$ satisfies
\begin{equation}
    -\nabla\cdot(\kappa(\mathbf{x})\nabla T(\mathbf{x}))
    =
    I_A(\mathbf{x})Q,
    \qquad
    \mathbf{x}\in D,
    \label{eq:heat_equation}
\end{equation}
where $\kappa(\mathbf{x})$ is the thermal conductivity, $Q=2000\,\mathrm{W/m^2}$,
and $I_A$ is the indicator of the heat source region
$A=(0.2,0.3)\,\mathrm{m}\times(0.2,0.3)\,\mathrm{m}$. A zero Neumann condition is
imposed on the top boundary, while zero Dirichlet conditions are imposed on the
remaining boundaries, as shown in Fig.~\ref{fig:heat_conduction}.

The thermal conductivity is modeled as a lognormal random field, $\kappa(\mathbf{x})
    =
    \exp(a_\kappa+b_\kappa f(\mathbf{x}))$,
where $f(\mathbf{x})$ is a standard Gaussian random field with covariance
\begin{equation}
    k(\mathbf{x},\mathbf{x}')
    =
    \exp\left(
    -\frac{\|\mathbf{x}-\mathbf{x}'\|_2^2}{l^2}
    \right),
    \qquad l=0.2 .
    \label{eq:heat_covariance}
\end{equation}
The constants $a_\kappa$ and $b_\kappa$ are chosen so that the mean and standard
deviation of $\kappa(\mathbf{x})$ are
$\mu_\kappa=1\,\mathrm{W}/({}^{\circ}\mathrm{C}\,\mathrm{m})$ and
$\sigma_\kappa=0.3\,\mathrm{W}/({}^{\circ}\mathrm{C}\,\mathrm{m})$, respectively.

We use the EOLE method \cite{betz2014numerical} to obtain a finite-dimensional
representation of the Gaussian random field. Let
$\{\boldsymbol{\xi}_i\}_{i=1}^{n}$ be a set of predefined grid points. The
Gaussian random field is approximated by $\widehat f(\mathbf{x})
    =
    \sum_{i=1}^{M}
    \frac{U_i}{\sqrt{l_i}}\,
    \boldsymbol{\phi}_i^{\top}
    \mathbf{C}_{\mathbf{x}\boldsymbol{\xi}}$,
 where $U_i$ are independent standard normal variables,
$\mathbf{C}_{\mathbf{x}\boldsymbol{\xi}}$ is the covariance vector with entries
$\left(\mathbf{C}_{\mathbf{x}\boldsymbol{\xi}}\right)_j
=
k(\mathbf{x},\boldsymbol{\xi}_j)$, and
$(l_i,\boldsymbol{\phi}_i)$ are the eigenvalue--eigenvector pairs of the
covariance matrix $\mathbf{C}_{\boldsymbol{\xi}\boldsymbol{\xi}}$. With grid
spacing $0.1\,\mathrm{m}$, we obtain $n=121$ grid points and retain $M=100$ EOLE
terms, which gives more than $99\%$ accuracy in the random field approximation.

The performance function is defined by the spatial average of the temperature
over the target region
$B=(-0.3,-0.2)\,\mathrm{m}\times(-0.3,-0.2)\,\mathrm{m}$:
\begin{equation}
    g(\mathbf{u})
    =
    8.5
    -
    \frac{1}{|B|}
    \int_{B}
    T(\mathbf{x};\mathbf{u})\,d\mathbf{x},
    \label{eq:heat_performance_function}
\end{equation}
where $\mathbf{u}=(U_1,\ldots,U_M)\in\mathbb{R}^{M}$. Failure occurs when
$g(\mathbf{u})\le 0$, that is, when the average temperature over $B$ exceeds
$8.5$. The forward problem is solved by the finite element method using $25040$
linear triangular elements; see Fig.~\ref{fig:heat_conduction}. A representative
input-output pair is shown in Fig.~\ref{fig:heat_solution}.

\begin{figure}[htbp]
    \centering
    \begin{overpic}[width=0.45\textwidth]{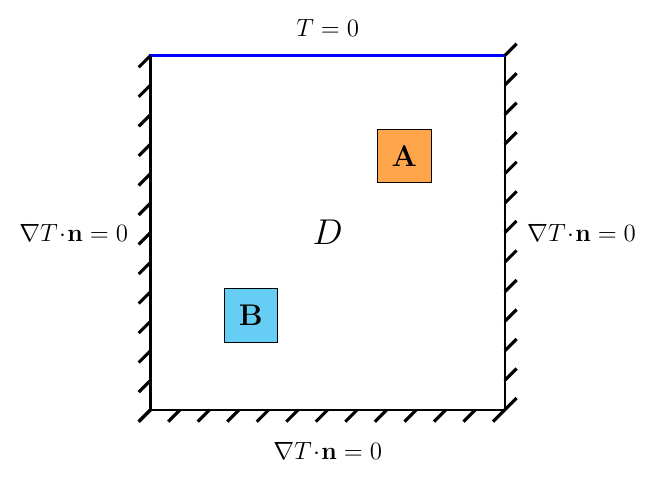}
    \end{overpic}
    \raisebox{.35cm}{
    \begin{overpic}[width=0.29\textwidth]{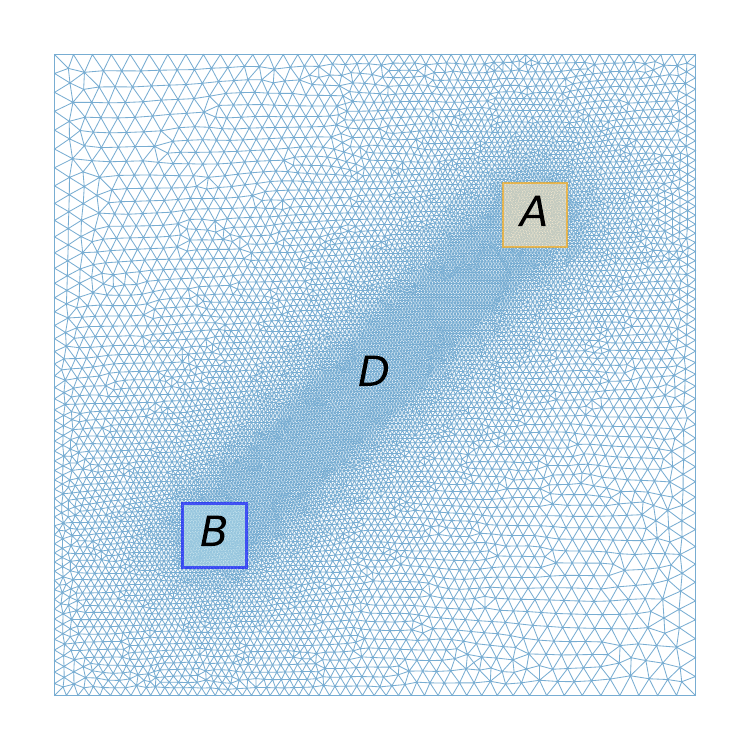}
    \end{overpic}
    }
    \vspace{-0.2cm}
    \caption{Heat conduction problem: computational domain and boundary
    conditions (left), and finite element mesh (right).}
    \label{fig:heat_conduction}
\end{figure}

\begin{figure}[htbp]
    \centering
    \includegraphics[width=0.75\linewidth]{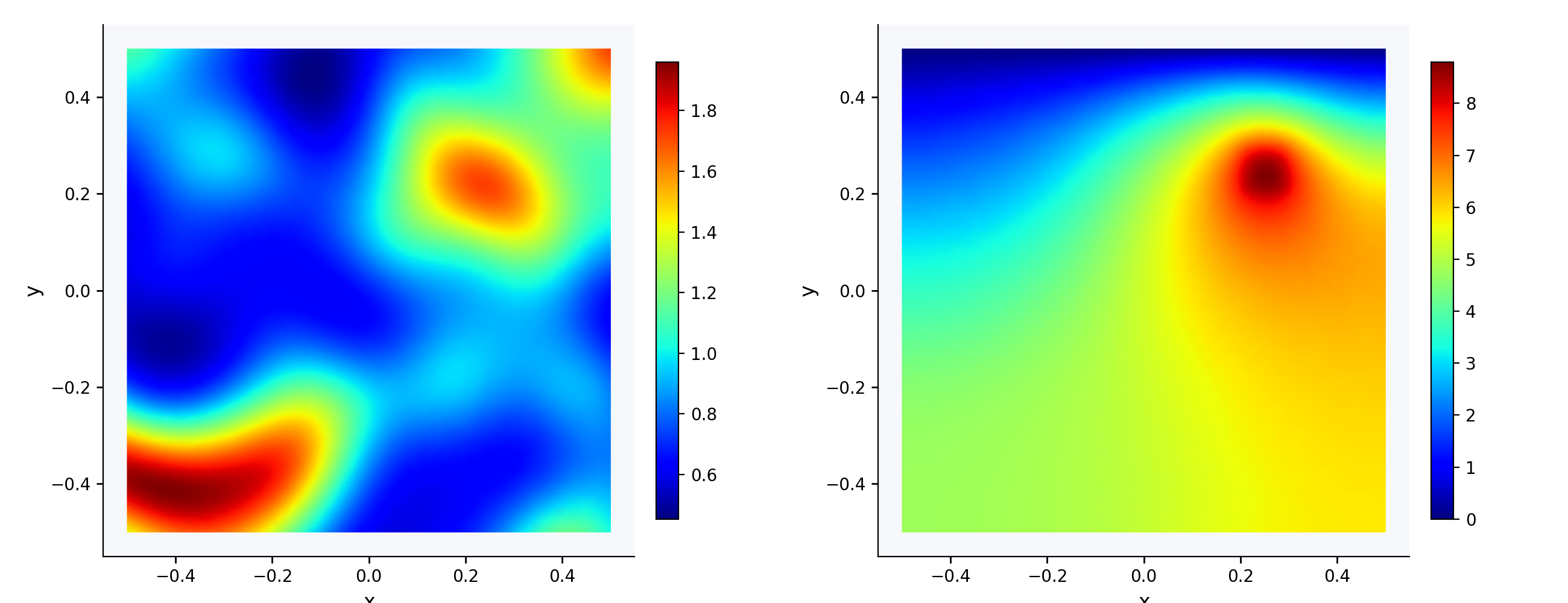}
    \vspace{-0.2cm}
    \caption{Representative input and output fields for the stochastic heat equation example. The left panel shows one realization of the random input field, while the right panel shows the corresponding temperature solution.}
    \label{fig:heat_solution}
\end{figure}

\begin{table}[htbp]
\centering
\caption{Performance comparison for the 100-dimensional stochastic
heat equation example. The reference value is obtained by the high-budget
ICE-vMFNM method.}
\label{tab:comparison_heat}
\begin{tabular}{lcccc}
\toprule
Method 
& $N_g$
& $\overline P_{\mathcal F}$
& $\varepsilon_{\mathrm{rel}}$
& $\delta$ \\
\midrule
ICE-vMFNM (reference)  
& $5.0\times 10^{4}$
& $7.63\times 10^{-5}$
& -- 
& 0.018 \\

ICE-vMFNM
& $4.2\times 10^3$
& $7.67\times 10^{-5}$
& $0.005$
& $0.098$ \\

PGGR-ICE-vMFNM
& $2.2\times 10^{2}$
& $7.57\times 10^{-5}$
& $0.008$
& $0.094$ \\

Random-ICE-vMFNM 

& $2.2\times 10^{2}$
& $4.39\times 10^{-5}$
& $0.425$
& $0.189$ \\ 
\bottomrule
\end{tabular}
\end{table}

Table~\ref{tab:comparison_heat} reports the performance comparison for the
100-dimensional stochastic heat equation example. The high-budget ICE-vMFNM
estimate is used as the reference value, giving
$P_{\mathcal F}^{\mathrm{ref}}=7.63\times 10^{-5}$. The reduced-budget
true-model ICE-vMFNM method gives a close estimate but still requires
$4.2\times 10^3$ actual evaluations of $g(\mathbf u)$. In comparison,
PGGR-ICE-vMFNM obtains a similarly accurate estimate,
$\overline P_{\mathcal F}=7.57\times 10^{-5}$, with only
$2.2\times 10^2$ high-fidelity evaluations on average, corresponding to an
approximately 19-fold reduction.

The Random-ICE-vMFNM baseline uses the same high-fidelity budget as
PGGR-ICE-vMFNM but substantially underestimates the failure probability. This
comparison indicates that the improvement is not merely due to adding a small
number of high-fidelity samples, but comes from the proposal-guided greedy
refinement strategy, which selects more informative samples in the
failure-relevant region.

\section{Conclusion}
\label{sec:conclusion}
We proposed a proposal-guided greedy surrogate refinement framework for
high-dimensional rare-event probability estimation. The method refines a
neural-network surrogate along the evolution of an adaptive importance sampling
proposal, avoiding the need for global surrogate accuracy over the full input
space. High-fidelity samples are selected near the surrogate-induced failure
boundary while maintaining diversity in a learned latent space, so that the
surrogate is improved in the region most relevant to proposal adaptation and
probability estimation.

The framework was instantiated with an ICE-vMFNM proposal update, and we analyzed
surrogate-induced misclassification, finite-sample estimation error, and one-step
proposal stability. These results show that the local surrogate error under the
proposal distribution controls both the surrogate-induced bias and the
perturbation of the proposal update.

Numerical experiments on multimodal, high-dimensional, and PDE-based rare-event
problems demonstrate that PGGR-ICE-vMFNM achieves accurate estimates with
substantially fewer high-fidelity evaluations than true-model adaptive importance
sampling. Future work will focus on multi-step error analysis and extensions to
more complex nonlinear and multimodal reliability problems.

\section*{Acknowledgments}
This work was supported by the MURI grant (FA9550-20-1-0358), the ONR Vannevar Bush Faculty Fellowship (N00014-22-1-2795), and the U.S. Department of Energy, Advanced Scientific Computing Research program, under the Scalable, Efficient and Accelerated Causal Reasoning Operators, Graphs and Spikes for Earth and Embedded Systems (SEA-CROGS) project (DE-SC0023191).
Additional funding was provided by GPU Cluster for Neural PDEs and Neural Operators to support MURI Research and Beyond, under Award \#FA9550-23-1-0671.

\appendix

\section{Proofs of the theoretical results}
\label{app:proofs}

\begin{proof}[Proof of Theorem~\ref{thm:one_step_proposal_stability}]
Since the derivative of \(\Phi\) is bounded by \(1/\sqrt{2\pi}\), we have
\[
    |h_{t+1}^\star(\mathbf u)-\widehat h_{t+1}(\mathbf u)|
    \le
    \frac{
    |g(\mathbf u)-\widehat g_t(\mathbf u)|
    }{
    \sigma_{t+1}\sqrt{2\pi}
    } .
\]
Changing measure from \(p\) to \(q_t\) and applying the Cauchy--Schwarz
inequality yields
\begin{align}
    \int
    |h_{t+1}^\star-\widehat h_{t+1}|p\,d\mathbf u
    &\le
    \frac{1}{\sigma_{t+1}\sqrt{2\pi}}
    \int
    |g-\widehat g_t|
    \frac{p}{q_t}
    q_t\,d\mathbf u
    \nonumber\\
    &\le
    \frac{\sqrt{C_w^{(2)}}}{\sigma_{t+1}\sqrt{2\pi}}
    \left\|
    g-\widehat g_t
    \right\|_{L^2(q_t)} .
    \label{eq:soft_indicator_l1_bound_app}
\end{align}
By the standard normalization inequality for probability densities,
\[
    \operatorname{TV}(\pi_{t+1}^\star,\widehat\pi_{t+1})
    \le
    \frac{2}{Z_{t+1}^\star}
    \int
    |h_{t+1}^\star-\widehat h_{t+1}|p\,d\mathbf u .
\]
Using \(Z_{t+1}^\star\ge z_{t+1}\) and
\eqref{eq:soft_indicator_l1_bound_app}, we obtain
\[
    \operatorname{TV}(\pi_{t+1}^\star,\widehat\pi_{t+1})
    \le
    \frac{
    2\sqrt{C_w^{(2)}}
    }{
    z_{t+1}\sigma_{t+1}\sqrt{2\pi}
    }
    \left\|
    g-\widehat g_t
    \right\|_{L^2(q_t)} .
\]
Finally, Assumption~\ref{ass:projection_stability} gives
\[
    \operatorname{TV}
    \left(
    q_{t+1}^{\star},
    \widehat q_{t+1}
    \right)
    \le
    L_{\mathcal Q}
    \operatorname{TV}(\pi_{t+1}^\star,\widehat\pi_{t+1}),
\]
which proves the result.
\end{proof}

\begin{proof}[Proof of Lemma~\ref{lem:surrogate_misclassification}]
For any \(0<\tau\le \tau_0\),
\[
    \mathcal M_t
    \subseteq
    \{|g(\mathbf u)|\le \tau\}
    \cup
    \{|g(\mathbf u)-\widehat g_t(\mathbf u)|>\tau\}.
\]
Therefore,
\[
    q_t(\mathcal M_t)
    \le
    q_t(|g(\mathbf u)|\le \tau)
    +
    q_t(|g(\mathbf u)-\widehat g_t(\mathbf u)|>\tau).
\]
By Assumption~\ref{ass:margin_condition},
\[
    q_t(|g(\mathbf u)|\le \tau)
    \le
    C_m\tau^\kappa .
\]
By Markov's inequality,
\[
    q_t(|g(\mathbf u)-\widehat g_t(\mathbf u)|>\tau)
    \le
    \frac{
    \|g-\widehat g_t\|_{L^2(q_t)}^2
    }{
    \tau^2
    }.
\]
Thus,
\[
    q_t(\mathcal M_t)
    \le
    C_m\tau^\kappa
    +
    \frac{
    \|g-\widehat g_t\|_{L^2(q_t)}^2
    }{
    \tau^2
    }.
\]
Optimizing the right-hand side with respect to \(\tau\) gives
\[
    q_t(\mathcal M_t)
    \le
    C
    \|g-\widehat g_t\|_{L^2(q_t)}^{\frac{2\kappa}{\kappa+2}},
\]
where \(C>0\) depends only on \(C_m\) and \(\kappa\).
\end{proof}

\begin{proof}[Proof of Theorem~\ref{thm:finite_sample_surrogate_error}]
We decompose the error as
\begin{align}
    \mathbb E
    \left[
    \left|
    \widehat P_{N,t}^{\mathrm{surr}}
    -
    P_{\mathcal F}
    \right|
    \right]
    &\le
    \left|
    P_{\widehat{\mathcal F},t}
    -
    P_{\mathcal F}
    \right|
    +
    \mathbb E
    \left[
    \left|
    \widehat P_{N,t}^{\mathrm{surr}}
    -
    P_{\widehat{\mathcal F},t}
    \right|
    \right].
    \label{eq:error_decomposition_app}
\end{align}

We first bound the surrogate-induced bias. Since
\[
    P_{\mathcal F}
    =
    \mathbb E_{q_t}
    \left[
    I(\mathbf u)\frac{p(\mathbf u)}{q_t(\mathbf u)}
    \right],
    \qquad
    P_{\widehat{\mathcal F},t}
    =
    \mathbb E_{q_t}
    \left[
    \widehat I_t(\mathbf u)\frac{p(\mathbf u)}{q_t(\mathbf u)}
    \right],
\]
we have
\begin{align}
    \left|
    P_{\widehat{\mathcal F},t}
    -
    P_{\mathcal F}
    \right|
    &\le
    \mathbb E_{q_t}
    \left[
    |\widehat I_t(\mathbf u)-I(\mathbf u)|
    \frac{p(\mathbf u)}{q_t(\mathbf u)}
    \right] \le
    C_w^{(\infty)}
    q_t(\mathcal M_t).
    \label{eq:bias_misclassification_app}
\end{align}
Applying Lemma~\ref{lem:surrogate_misclassification}, we obtain
\[
    \left|
    P_{\widehat{\mathcal F},t}
    -
    P_{\mathcal F}
    \right|
    \le
    C
    \|g-\widehat g_t\|_{L^2(q_t)}^{\frac{2\kappa}{\kappa+2}},
\]
where \(C>0\) depends on \(C_m\), \(\kappa\), and
\(C_w^{(\infty)}\). We next bound the finite-sample term. Let
\[
    Y_t(\mathbf u)
    =
    \widehat I_t(\mathbf u)\frac{p(\mathbf u)}{q_t(\mathbf u)},
    \qquad
    \mathbf u\sim q_t .
\]
Then $\mathbb E_{q_t}[Y_t]=P_{\widehat{\mathcal F},t}$,
and $\widehat P_{N,t}^{\mathrm{surr}}
    =
    \frac{1}{N}
    \sum_{i=1}^N Y_t(\mathbf u_i)$.
By Jensen's inequality,
\[
    \mathbb E
    \left[
    \left|
    \widehat P_{N,t}^{\mathrm{surr}}
    -
    P_{\widehat{\mathcal F},t}
    \right|
    \right]
    \le
    \sqrt{
    \operatorname{Var}
    \left(
    \widehat P_{N,t}^{\mathrm{surr}}
    \right)
    }
    =
    \frac{1}{\sqrt N}
    \sqrt{\operatorname{Var}_{q_t}(Y_t)}.
\]
Using the surrogate-induced zero-variance density $\widehat q_t^\star(\mathbf u)$, 
we compute
\begin{align}
    \operatorname{Var}_{q_t}(Y_t)
    &=
    \int
    \widehat I_t(\mathbf u)
    \frac{p(\mathbf u)^2}{q_t(\mathbf u)}
    d\mathbf u
    -
    P_{\widehat{\mathcal F},t}^2
    \nonumber\\
    &=
    P_{\widehat{\mathcal F},t}^2
    \left[
    \int
    \frac{
    \widehat q_t^\star(\mathbf u)^2
    }{
    q_t(\mathbf u)
    }
    d\mathbf u
    -
    1
    \right]
    \nonumber\\
    &=
    P_{\widehat{\mathcal F},t}^2
    \chi^2(\widehat q_t^\star\|q_t).
    \label{eq:variance_chi_square_app}
\end{align}
Therefore,
\[
    \mathbb E
    \left[
    \left|
    \widehat P_{N,t}^{\mathrm{surr}}
    -
    P_{\widehat{\mathcal F},t}
    \right|
    \right]
    \le
    \frac{
    P_{\widehat{\mathcal F},t}
    }{
    \sqrt{N}
    }
    \sqrt{
    \chi^2(\widehat q_t^\star\|q_t)
    } .
\]
Combining this estimate with
\eqref{eq:error_decomposition_app} and
\eqref{eq:bias_misclassification_app} gives
\eqref{eq:finite_sample_chi_square_bound}.

Finally, since $P_{\widehat{\mathcal F},t}
    \le
    P_{\mathcal F}
    +
    \left|
    P_{\widehat{\mathcal F},t}
    -
    P_{\mathcal F}
    \right|$, 
and the surrogate-induced probability error is bounded by the first term above,
we obtain
\eqref{eq:finite_sample_non_circular_bound}. This completes the proof.
\end{proof}

\section{Ablation study}
\label{sec:ablation_study}

We further investigate the sensitivity of the proposed method to two parameters
in the greedy refinement step: the diversity weight \(\beta\) and the number of
new high-fidelity samples \(m_{\mathrm{add}}\) selected at each adaptive
iteration. All ablation tests are performed on the 100-dimensional stochastic
diffusion example, while the remaining settings are kept the same as in
Section~\ref{subsec:stochastic_diffusion_example}. During adaptive refinement, the last
encoder layer is kept fixed so that the latent-space metric used in the greedy
selection rule remains stable.

We first study the effect of the diversity weight \(\beta\). In this test, we
fix \(m_{\mathrm{add}}=70\) and vary \(\beta\). As shown in
Fig.~\ref{fig:beta_sensitivity}, the choice of \(\beta\) affects both the bias
and the stability of the estimator. When \(\beta\) is too small, the selected
samples may concentrate near a limited part of the estimated failure boundary,
leading to insufficient coverage and possible underestimation of the failure
probability. As \(\beta\) increases, the diversity term improves the coverage of
the proposal-induced important region. However, an overly large \(\beta\) may
make the selected samples too dispersed, which can weaken boundary refinement
and increase the coefficient of variation. A moderate value, such as
\(\beta=0.5\), provides a stable balance between boundary proximity and
latent-space diversity.

\begin{figure}[t]
    \centering 
    \includegraphics[width = 0.75\textwidth]{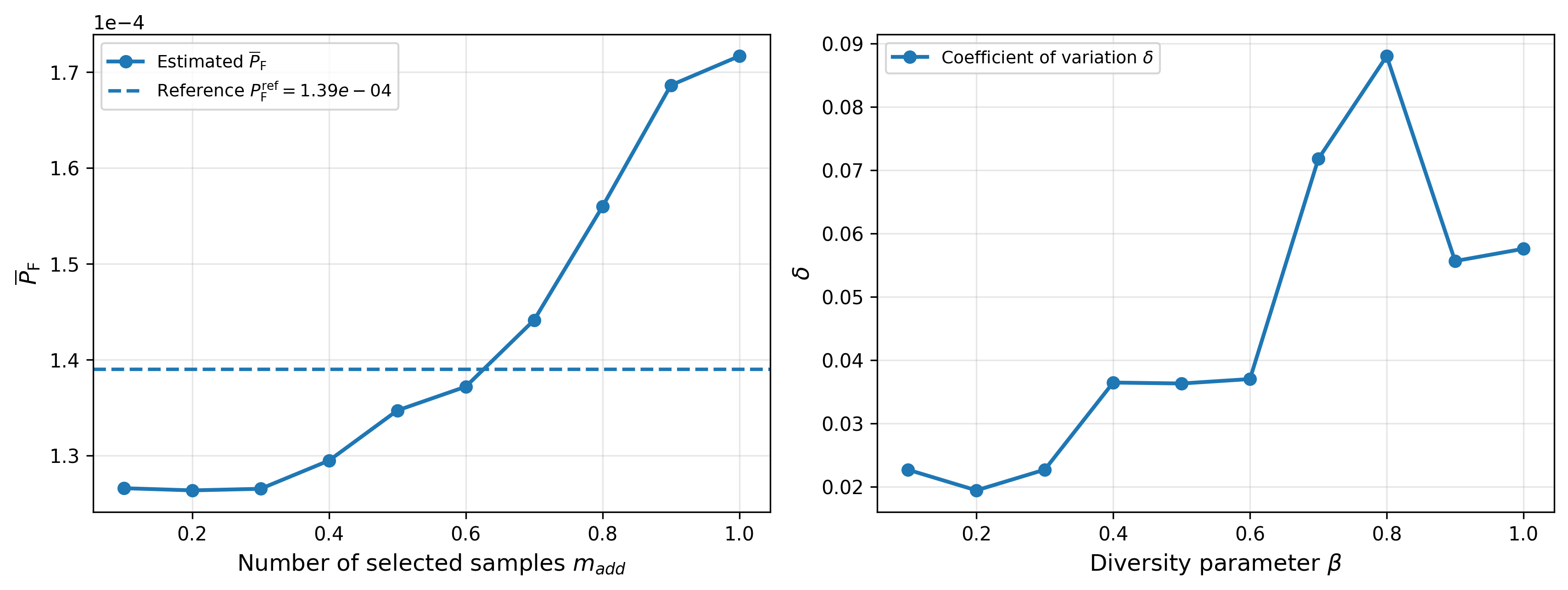}
    \vspace{-0.2cm}
    \caption{Sensitivity of the estimated failure probability and coefficient of
    variation to the diversity parameter \(\beta\) for the 100-dimensional
    stochastic diffusion example. Here \(m_{\mathrm{add}}=70\) is fixed.}
    \label{fig:beta_sensitivity}
\end{figure}

We next study the effect of \(m_{\mathrm{add}}\). In this test, we fix
\(\beta=0.5\) and vary the number of newly selected high-fidelity samples per
adaptive iteration. Figure~\ref{fig:madd_sensitivity} shows that increasing
\(m_{\mathrm{add}}\) generally improves the accuracy and reduces the coefficient
of variation. This is expected because more selected high-fidelity samples
provide stronger local correction of the surrogate near the proposal-induced
failure boundary. The improvement, however, comes with a higher evaluation cost.
In the reported high-dimensional experiments, \(m_{\mathrm{add}}=70\) gives a
reasonable accuracy--cost trade-off.

\begin{figure}[htbp]
    \centering 
    \includegraphics[width = 0.75\textwidth]{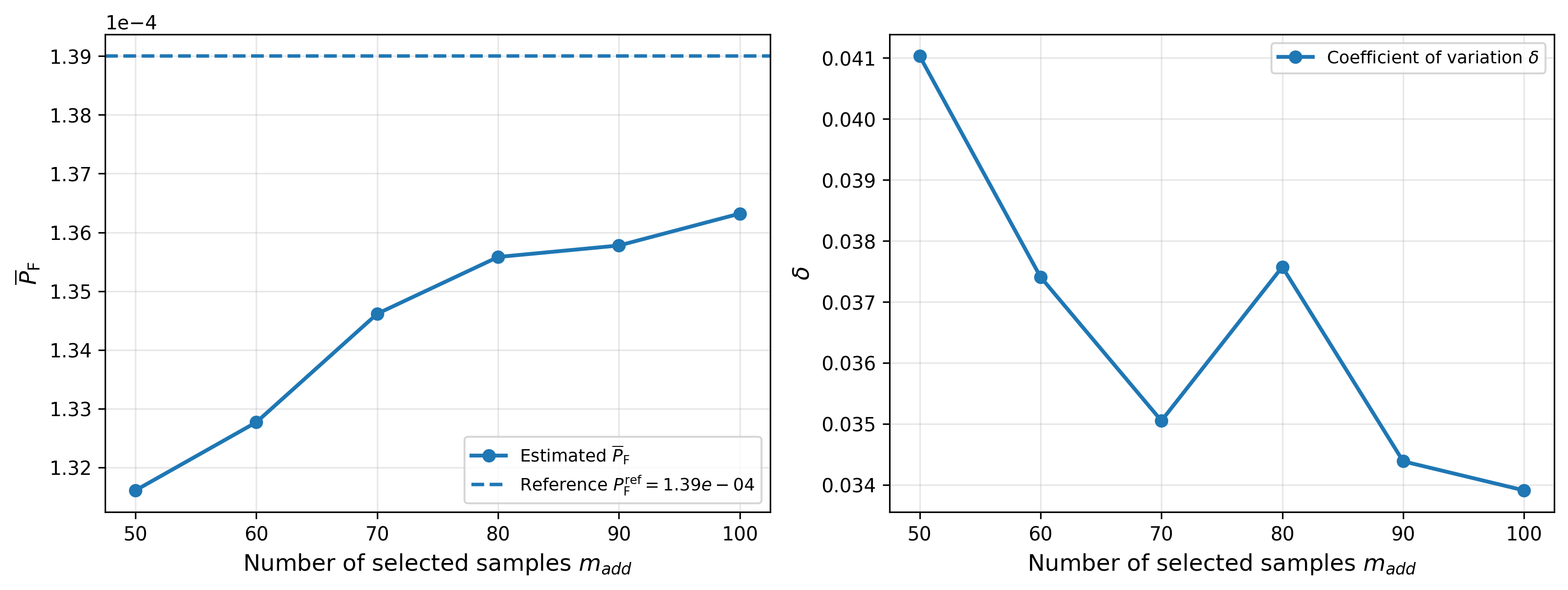}
    \vspace{-0.2cm}
    \caption{Sensitivity of the estimated failure probability and coefficient of
    variation to the number of selected high-fidelity samples
    \(m_{\mathrm{add}}\) for the 100-dimensional stochastic diffusion example.
    Here \(\beta=0.5\) is fixed.}
    \label{fig:madd_sensitivity}
\end{figure}

\bibliographystyle{unsrt}
\bibliography{references}

\end{document}